\documentclass[11pt]{elsarticle}

\usepackage{lineno,hyperref}

\usepackage{float}
\usepackage{graphicx}
\usepackage{subfig}
\usepackage{mathrsfs}
\usepackage{graphics}
\usepackage{float}
\usepackage{amsmath, commath}
\usepackage{amsfonts}
\usepackage{mathrsfs}
\usepackage{amsmath}
\usepackage[table]{xcolor}

\usepackage{optidef}

\usepackage{optidef}
\usepackage{graphicx}
\usepackage{float}
\usepackage{graphicx}
\usepackage{subfig}
\usepackage{mathrsfs}
\usepackage{float}
\usepackage{amsmath}
\usepackage{amsfonts}
\usepackage{booktabs}
\usepackage{natbib}
\usepackage{bbm}
\usepackage{amsmath}
\usepackage{algorithm}
\usepackage[noend]{algpseudocode}
\usepackage{multirow}
\algdef{SE}[DOWHILE]{Do}{doWhile}{\algorithmicdo}[1]{\algorithmicwhile\ #1}

\usepackage[letterpaper, portrait, margin=1in]{geometry}

\usepackage{booktabs}
\usepackage{tabularx}
\usepackage{setspace}

\usepackage{xcolor}
\hypersetup{
    colorlinks,
    linkcolor={red!50!red},
    citecolor={blue!50!blue},
    urlcolor={blue!80!blue}
}

\usepackage{optidef}

\usepackage{amsthm}
\newtheorem{theorem}{Theorem}
\newtheorem{theorem2}{Theorem2}
\newtheorem{prop}{Proposition}
\theoremstyle{definition}

\newtheorem{lemma}[theorem]{Lemma}
\newtheorem{remark}[theorem2]{Remark}

\newcommand{\bx}{\boldsymbol{x}}
\newcommand{\bb}{\boldsymbol{\beta}}
\makeatletter

\usepackage{newtxtext} % use times new roman font style

\DeclareMathOperator*{\argmax}{arg\,max}
\DeclareMathOperator*{\argmin}{arg\,min}

\makeatletter
\newcommand*{\centerfloat}{%
  \parindent \z@
  \leftskip \z@ \@plus 1fil \@minus \marginparwidth
  \rightskip \leftskip
  \parfillskip \z@skip}
\makeatother

\usepackage{amssymb,amsmath,caption}
\usepackage[hang,flushmargin]{footmisc}
\newcommand{\algorithmfootnote}[2][\footnotesize]{
  \let\old@algocf@finish\@algocf@finish% Store algorithm finish macro
  \def\@algocf@finish{\old@algocf@finish% Update finish macro to insert "footnote"
    \leavevmode\rlap{\begin{minipage}{\linewidth}
    #1#2
    \end{minipage}}%
  }%
}

\usepackage{setspace}
\usepackage{natbib}
\usepackage{algorithm}
\usepackage[noend]{algpseudocode}
\algdef{SE}[DOWHILE]{Do}{doWhile}{\algorithmicdo}[1]{\algorithmicwhile\ #1}

\usepackage{appendix}
\usepackage{array}
\usepackage{threeparttable}
\usepackage{multirow}
\usepackage{rotating}
\usepackage{makecell}

\journal{}

\biboptions{authoryear}

\begin{document}
\begin{frontmatter}

\title{Robust Binary and Multinomial Logit Models for Classification with Data Uncertainties}
% \title{Appendix: Methodology for probabilistic Inference}


\author[label0,label1]{Baichuan Mo}
\author[label1,label2]{Yunhan Zheng\corref{mycorrespondingauthor}}
\author[label1]{Xiaotong Guo}
\author[label3]{Ruoyun Ma}
\author[label4]{Jinhua Zhao}
\address[label0]{Department of Civil Engineering, Tsinghua University, Beijing, China, 100084}
\address[label1]{Department of Civil and Environmental Engineering, Massachusetts Institute of Technology, Cambridge, MA 02139}
\address[label2]{Singapore–MIT Alliance for Research and Technology Centre (SMART), Singapore}
\address[label3]{Department of Management Science and Engineering, Stanford University, Stanford, CA 94305}
\address[label4]{Department of Urban Studies and Planning, Massachusetts Institute of Technology, Cambridge, MA 20139}

\cortext[mycorrespondingauthor]{Corresponding author}

\begin{abstract}
Binary logit (BNL) and multinomial logit (MNL) models are the two most widely used discrete choice models for travel behavior modeling and prediction. However, in many scenarios, the collected data for those models are subject to measurement errors. Previous studies on measurement errors mostly focus on ``better estimating model parameters'' with training data. In this study, we focus on using BNL and MNL for classification problems, that is, to ``better predict the behavior of new samples'' when measurement errors occur in testing data. To this end, we propose a robust BNL and MNL framework that is able to account for data uncertainties in both features and labels. The models are based on robust optimization theory that minimizes the worst-case loss over a set of uncertainty data scenarios. Specifically, for feature uncertainties, we assume that the $\ell_p$-norm of the measurement errors in features is smaller than a pre-established threshold. We model label uncertainties by limiting the number of mislabeled choices to at most $\Gamma$. Based on these assumptions, we derive a tractable robust counterpart. The derived robust-feature BNL and the robust-label MNL models are exact. However, the formulation for the robust-feature MNL model is an approximation of the exact robust optimization problem. An upper bound of the approximation gap is provided. We prove that the robust estimators are inconsistent but with a higher trace of the Fisher information matrix. They are preferred when out-of-sample data has errors due to the shrunk scale of the estimated parameters. The proposed models are validated in a binary choice data set and a multinomial choice data set, respectively. Results show that the robust models (both features and labels) can outperform the conventional BNL and MNL models in prediction accuracy and log-likelihood. We show that the robustness works like ``regularization'' and thus has better generalizability.

\end{abstract}


\end{frontmatter}

% \linenumbers % LINE NUMBER
\section{Introduction}
Binary logit (BNL) and multinomial logit (MNL) models are the two most widely used discrete choice models (DCMs) \citep{ben1985discrete}. They are widely used to describe, explain, and predict choices between two or more discrete alternatives, such as entering or not entering the labor market, or choosing between modes of transport. The models specify the probability that a person chooses a particular alternative, with the probability expressed as a function of observed variables that relate to the alternatives and the person.

\subsection{Preliminaries}
BNL and MNL models can be derived from random utility theory. Particularly, let $\bx_{n,i}$ be the vector of observed factors for person $n$ with respective alternative $i$, defined as \begin{align}
    \bx_{n,i} = (x_{n,i}^{(k)})_{k\in\mathcal{K}_i}, \quad \forall i \in \mathcal{C}_n,
\end{align} 
where $x_{n,i}^{(k)}$ is the $k$-th element of $\bx_{n,i}$. $\mathcal{C}_n$ is the set of available alternatives for person $n$ (for example, public transit, walking, car). The set of all alternatives is $\mathcal{C} := \cup_{n\in\mathcal{N}} \mathcal{C}_n$. $\mathcal{K}_i$ is the set of features for alternative $i$. The set of all features is $\mathcal{K} := \cup_{i\in\mathcal{C}} \mathcal{K}_i$. For example, features may include travel time, travel cost, and a person's income.  $\mathcal{C}$ and $\mathcal{K}$ are both finite.

Let $U_{n,i}$ be the utility that person $n$ obtains from choosing alternative $i$. The person's utility depends on many factors, some of which we observe and some not. Hence, $U_{n,i}$ can be decomposed into a part that depends on observed variables and another part with unobserved factors. In a linear form, the utility is expressed as
\begin{align}\label{eq_utility}
    U_{n,i} = \boldsymbol{\beta}_i^T  \boldsymbol{x}_{n,i} + \epsilon_{n,i},
\end{align}
where $\boldsymbol{\beta}_i \in\mathbb{R}^{|\mathcal{K}_i|}$ is the corresponding vector of coefficients to be estimated. It represents the contribution of a feature (e.g., cost, travel time) to the total utility. $\epsilon_{n,i}$ is a random variable that captures the impact of all unobserved factors that affect the person's choice and zero-mean measurement errors. $U_{n,i}$ is also a random variable. 

The choice of the person is designated by dummy variables $y_{n,i}$. $\boldsymbol{y_n} = (y_{n,i})_{i \in \mathcal{C}_n}$ is the associated vector. $y_{n,i} = 1$ indicating person $n$ choosing alternative $i$ and $y_{n,i} = 0$ otherwise. Based on the utility maximization assumption (i.e., people will choose the highest utility alternative), the probability of $y_{n,i} = 1$ can be modeled as 
\begin{align}\label{eq_mnl_p}
    P_{n,i} := \mathbb{P}(y_{n,i} = 1) &= \mathbb{P}(U_{n,i} \geq U_{n,j}, \; \forall j\in \mathcal{C}_{n}\setminus\{i\}) \notag \\
    & = \mathbb{P}\big(\epsilon_{n,i}-\epsilon_{n,j} \geq (\boldsymbol{\beta}_j^T  \boldsymbol{x}_{n,j} - \boldsymbol{\beta}_i^T  \boldsymbol{x}_{n,i}), \; \forall j\in \mathcal{C}_{n}\setminus\{i\}\big) .
\end{align}
In BNL and MNL models, $\epsilon_{n,i}$ is assumed to be independent and identically Gumbel distributed. Then $\epsilon_{n,i} - \epsilon_{n,j}$ will follow the logistic distribution with cumulative density function (CDF) $F_{\Delta \epsilon_{i,j}}(z) = \frac{1}{1+e^{-z}}$. When $|\mathcal{C}_n| =2$ for all people, this formulation represents the BNL model. Eq. \ref{eq_mnl_p} has a closed-form formulation \citep{train2009discrete}:
\begin{align}
    P_{n,i}  = F_{\Delta \epsilon_{i,j}}(\boldsymbol{\beta}_i^T  \boldsymbol{x}_{n,i} - \boldsymbol{\beta}_j^T  \boldsymbol{x}_{n,j}) = \frac{1}{1+e^{\boldsymbol{\beta}_j^T \boldsymbol{x}_{n,j} - \boldsymbol{\beta}_i^T \boldsymbol{x}_{n,i}}} = \frac{e^{\boldsymbol{\beta}_i^T \boldsymbol{x}_{n,i}}}{ e^{\boldsymbol{\beta}_i^T \boldsymbol{x}_{n,i}}+e^{\boldsymbol{\beta}_j^T \boldsymbol{x}_{n,j}}}, \quad \forall i,j\in\mathcal{C}_n,
\end{align}
which is equivalent to the canonical logistic regression if $\boldsymbol{x}_{n,j}=\boldsymbol{x}_{n,i}$. 

For $|\mathcal{C}_n| > 3$, the formulation represents the MNL model and Eq. \ref{eq_mnl_p} is equivalent to $\mathbb{P}(U_{n,i} - \max_{j\in\mathcal{C}_{n}\setminus\{i\}}U_{n,j} \geq 0)$. Given the property of the Gumbel distribution, $U_n^* := \max_{j\in\mathcal{C}_{n}\setminus\{i\}}U_{n,j}$ is also Gumbel distributed with the location parameter equal to $\sum_{j\in\mathcal{C}_{n}\setminus\{i\}} \boldsymbol{\beta}_j^T \boldsymbol{x}_{n,j}$. Then the CDF of $(U_{n,i} - \max_{j\in\mathcal{C}_{n}\setminus\{i\}}U_{n,j})$ will be logistic distributed \citep{ben1985discrete} with CDF $F_{\Delta U_{i,*}}(z) = \frac{1}{1+\exp(-z + \bb_i^T \bx_{n,i} - {\sum_{ j\in\mathcal{C}_{n}\setminus\{i\}}}\bb_j^T \bx_{n,j})}$. Therefore, the probability for the MNL model is 
\begin{align}
    P_{n,i} = \mathbb{P}(U_{n,i} - \max_{j\in\mathcal{C}_{n}\setminus\{i\}}U_{n,j} \geq 0) =1- F_{\Delta U_{i,*}}(0) =\frac{e^{\boldsymbol{\beta}_i^T \boldsymbol{x}_{n,i}}}{\sum_{j\in\mathcal{C}_n} e^{\boldsymbol{\beta}_j^T \boldsymbol{x}_{n,j}}},
\end{align}

Maximum likelihood estimation (MLE) is usually used to estimate the BNL and MNL models to get the parameter. Define $\boldsymbol{\beta} := (\bb_j)_{j\in\mathcal{C}}$. The MLE can be expressed as the following optimization problem: 
\begin{align}
    \max_{\boldsymbol{\beta}} & \sum_{n\in\mathcal{N}} \sum_{i\in\mathcal{C}_n} y_{n,i} \cdot \log(P_{n,i}(\boldsymbol{\beta})) = \max_{\boldsymbol{\beta}} \sum_{n\in\mathcal{N}} \sum_{i\in\mathcal{C}_n} y_{n,i} \cdot \log\left(\frac{\exp({\boldsymbol{\beta}_i^T \boldsymbol{x}_{n,i}})}{\sum_{j\in\mathcal{C}_n} \exp({\boldsymbol{\beta}_j^T \boldsymbol{x}_{n,j}})}\right),
    \label{eq_mle}
\end{align}
where $\mathcal{N}$ is the set of all persons. 

\subsection{Uncertainties in data}
BNL and MNL models are usually used as classifiers to predict an individual's behavior based on the estimated parameters from real-world data (such as surveys). However, in many scenarios, the collected data are subject to uncertainties (such as erroneous responses, dictation errors, etc.), which are known as measurement errors \citet{hausman2001mismeasured}. Beyond measurement errors, there are other sources of data uncertainty in transportation. For instance, \citet{li2023modified} identify three key factors: (a) perceived data used to construct models can be perturbed due to errors in measurement or recording; (b) while nominal distributions are used for random parameters in transportation networks, actual distributions may deviate, such as when roads unexpectedly close or maintenance work extends longer than planned; and (c) future validation data may differ from past data on which the model was built. Additionally, \citet{mo2022quantifying} highlighted other sources of data uncertainty in traffic state estimation, including inherent stochasticity in driving behavior and random initial or boundary conditions. This research specifically focuses on measurement error issues within the context of DCMs.

Data errors may lead to biased or inconsistent estimates of model parameters, deteriorating the model's predictive power. Note that although the error term $\epsilon_{n,i}$ in Eq. \ref{eq_utility} is assumed to capture some zero-mean measurement errors, an important assumption for the error term is that $\epsilon_{n,i}$ is uncorrelated with observed factors $\bx_n$. However, in reality, measurement errors are often caused by misreporting some specific factors (like travel time, income, etc.), which are inevitably related to $\bx_n$. Therefore, relying solely on $\epsilon_{n,i}$ to capture measurement errors will break the independence assumption and cause endogeneity. 

Errors can happen in features (i.e., $\bx_n$, left-hand side variables) and labels (i.e., $\boldsymbol{y}_n$, right-hand side variables), which require different ways to address. The typical way to deal with Measurement errors in the literature is instrumental variables \citep{hausman2001mismeasured}. The instrumental variables are assumed to be correlated with the ``true value'' of the mismeasured variables but uncorrelated with error terms. 

Previous literature on measurement errors usually focuses on ``better estimating model parameters'' with training data. However, in the real world when using the trained (or estimated) DCMs to \textbf{predict} new users (or samples) behavior (i.e., a classification problem), measurement errors also exist in the testing data set. This implies that an ``unbiased estimation'' after correcting bias may end up performing worse in the travel behavior prediction task. Our paper will focus on this task, which differentiates us from previous econometrics studies, as shown in Table \ref{tab_compare}. Specifically, our study assumes uncertainties in the testing data set and aims to improve prediction accuracy under uncertainties. While previous studies assume uncertainties in the training data set and focus on estimating unbiased model parameters with the training data. This makes our paper closely related to developing robust classifiers \citep{bertsimas2019robust}. 

\begin{table}[H]
\centering
\caption{Comparison of this study and literature}\label{tab_compare}
\begin{tabular}{cccc}
\hline
 & \textbf{Task} & \textbf{Performance} & \textbf{Data uncertainties}  \\
\hline

\textbf{Literature}           & Estimate unbiased parameters            & Interpretability            & Training data                       \\
\textbf{This study}         & Predict new samples            & Prediction accuracy            & Training and testing data                     \\
\hline
\end{tabular}
\end{table}

\subsection{Organization and contributions}
In this paper, we propose a robust discrete choice model framework that is able to account for data uncertainties in both features and labels. The objective is to provide a more accurate prediction of an individual's behavior for new samples (i.e., testing data set) as a classification task when there existing data errors. The model is based on a robust optimization framework that minimizes the worst-case loss over a set of uncertainty data scenarios. Specifically, for feature uncertainties, we assume that the $\ell_p$-norm of the measurement errors in features is smaller than a pre-established threshold. We model label uncertainties by limiting the number of mislabeled choices to at most $\Gamma$. Based on these assumptions, we derive a tractable robust counterpart for robust-feature and robust-label DCM models. The derived robust-feature BNL and the robust-label MNL models are exact. However, the formulation for the robust-feature MNL model is an approximation of the exact robust optimization problem. The proposed models are validated in a binary choice data set and a multinomial choice data set, respectively. Results show that the robust models (both features and labels) can outperform the conventional BNL and MNL models in prediction accuracy and log-likelihood. We show that the robustness works like ``regularization'' and thus has better generalizability. 

The main contribution of the paper is as follows:
\begin{itemize}
    \item We drive the closed-form robust counterparts for the robust BNL and MNL models. Specifically, the formulations for robust-feature BNL robust-label MNL models are exact (i.e., not an approximation). For robust MNL, due to the difficulties in maximizing a convex function, we use Jensen's inequality to approximate the original objective function, yielding a lower bound of the original maximization problem. We prove that the gap of the approximation will not be extreme. 

    \item We prove that the proposed robust-feature and robust-label estimators are inconsistent, but they have a higher trace of the Fisher information matrix compared to the MLE estimator (usually imply lower variance). When predicting new samples with data errors, robust estimators are preferred due to the shrinking scale of the estimated parameters. 
    \item We explain the good performance of robust DCM models from the aspects of both theories and experiments. We show that robust models tend to shrink the scale of the estimated parameters, which lowers the expected prediction errors when testing data has perturbations. Experiment results also validate the theories. 
\end{itemize}

Our work is connected to the robust classification methods proposed by \citet{bertsimas2019robust}, where robust optimization and logistic regression are integrated to construct classifiers that are capable of handling uncertainties in both features and labels. We draw inspiration from these approaches in terms of constructing suitable uncertainty sets and deriving robust counterparts in robust classifications. However, our work goes beyond their exclusive focus on binary classification. Instead, we contribute by developing robust DCMs for more than two categories and providing a theoretical analysis of the statistical properties of the robust estimators. For more than categories, we show that no exact robust counterpart exists, and an outer approximation is needed. We also analyze the quality of the approximation and prove that the gap of the approximation will not be extreme. 

The remainder of this paper is organized as follows. The literature review is presented in Section \ref{sec_liter}. In Section \ref{sec_feature_uncertainty}, we describe formulations and derivations of the robust-feature DCMs for binary and multinomial cases. In Section \ref{sec_label}, we elaborate on the robust-label DCMs that are suitable for both binary and multinomial cases. Section \ref{sec_stats} discusses the theoretical statistical properties and the out-of-sample prediction performance for the robust estimators. We apply the proposed framework to two different data sets as case studies in Section \ref{sec_case_study}. Conclusions and discussions are presented in Section \ref{sec_conclusion}.

\section{Literature review}\label{sec_liter}
\subsection{Measurement errors in DCM}
Transportation planning and policy analysis heavily rely on travel survey data, which includes information about activity patterns, travel behaviors, and comprehensive socio-demographic profiles of the surveyed populations. However, it's crucial to acknowledge that the presence of measurement errors in survey data is not uncommon. These errors can affect various travel-related variables, including mode choice, trip duration, and travel costs, as well as socio-demographic factors such as income \citep{paleti2019misclassification}. These errors present a significant challenge when utilizing DCM for the analysis of travel surveys. For instance, a Monte Carlo simulation conducted by \citet{hausman1998misclassification} revealed that even a small rate of outcome misclassifications (e.g., 2\%) can result in DCM estimates that exhibit biases ranging from 15\% to 25\% when the outcome is binary.

The origins of measurement errors in household surveys are multifaceted and can be attributed to several distinct sources. First, respondents may consciously choose to provide misleading information due to various motivations, including the desire to conceal certain details or to offer socially acceptable responses \citep{kreuter2008social}. For instance, research has revealed that self-employed individuals, in particular, may deliberately underreport their income by as much as 25\% when participating in household surveys \citep{hurst2014household}. Second, measurement errors can stem from inadvertent misreporting by respondents. This occurs when individuals encounter difficulties in comprehending survey questions, struggle to recollect specific details from memory, or employ inappropriate decision-making heuristics \citep{campanelli1991use}. Furthermore, the precision of survey data can be intricately connected to the particular survey methods and tools utilized for data collection. For instance, individuals engaged in surveys conducted through Computer Assisted Telephone Interviewing (CATI) methodologies have exhibited systematic tendencies to underreport travel, underestimate travel distances, and overstate travel times, in comparison to surveys that leverage GPS-based tracking technology \citep{stopher2007assessing}.

DCM has been widely used for both understanding people's travel decisions and conducting activity-travel planning \citep{ben1999discrete,bowman2001activity}. Consequently, the presence of measurement errors within the data can potentially distort policy decisions that rely on estimation outcomes derived from compromised data. For instance, when essential travel decision factors like travel cost and time are inaccurately measured, it can yield erroneous estimates of their influence on travelers' preferences. Likewise, if mode choice is subject to mismeasurement, it may result in investments in transportation infrastructure that do not align with the genuine preferences of travelers, potentially leading to suboptimal resource allocation and inefficient utilization of public funds. Therefore, it is imperative to develop methods that account for data mismeasurements within traditional DCM.

Data perturbation and uncertainty are not limited to measurement errors in DCM but extend to other fields of transportation modeling. For instance,  in the field of traffic modeling, \citet{li2023modified} identify several potential sources of data uncertainty, such as measurement and recording errors, deviations between nominal and actual distributions of random parameters, and discrepancies between future validation data and the historical data used for model development. They demonstrate the robustness of their Modified Late Arrival Penalized User Equilibrium (MLAPUE) model in addressing these data perturbations. \citet{shao2021optimization} emphasize the critical impact of sensor measurement error, noting that magnetic field interference can compromise the accuracy of inductive loop detectors, while factors such as visibility, lighting, and weather conditions can significantly affect video detector efficiency—ultimately influencing traffic flow estimation.  In addition to measurement error, traffic stochasticity is another source of uncertainty. \citet{mo2022quantifying} underscore the inherent randomness in driving behaviors and propose a Physics-Informed Generative Adversarial Network (GAN) to effectively quantify uncertainty in traffic state estimation.  \citet{musunuru2019applications} highlight the serious implications of data measurement errors in road safety modeling. These errors often stem from transportation agencies extrapolating short-term traffic counts over time and space, leading to biases in regression coefficient estimates and increased model dispersion. To address this, they propose robust measurement error correction techniques, including regression calibration and simulation extrapolations.

The widespread occurrence of data uncertainty in transportation modeling underscores its critical impact on the reliability and robustness of analytical results. While numerous studies have explored the implications of measurement errors and uncertainty in traffic modeling, the challenges posed by data uncertainty in DCM have received relatively limited attention. This study addresses this gap by focusing specifically on the effects of measurement errors in DCM, thereby contributing to the advancement of robust modeling practices within the broader transportation research field.

% There have been relatively few studies addressing the issue of data mismeasurement errors in travel behavior modeling. 

\subsection{Methods for addressing measurement errors}

To address the issue of feature mismeasurement (or ``uncertainty''), researchers have primarily employed two main coping strategies \citep{schennach2016recent}. The first approach centers on the recovery of accurate, error-free values from data tainted by measurement errors. However, this method assumes prior knowledge of the measurement error distribution, which may not always align with real-world situations. For instance, some earlier studies employed Fourier transform algorithms to mitigate measurement errors while assuming the availability of known error distributions \citep{wang2011deconvolution,schennach2019convolution}.

The second strategy involves correcting measurement error biases by incorporating readily available auxiliary variables, which can include repeated measurements \citep{schennach2004estimation}, indicators \citep{ben2014identification}, or instrumental variables \citep{hu2008identification}. The instrumental variable approach, in particular, has been widely adopted to mitigate mismeasurement problems in linear specifications \citep{hausman2001mismeasured}. In this approach, instrumental variables are carefully selected to serve as proxies for the imperfectly measured feature variables. They are chosen based on their lack of correlation with the measurement errors, allowing them to effectively separate the measurement errors from the estimation process for the dependent variable \citep{baiocchi2014instrumental}. However, a significant challenge with this group of approaches is the difficulty of obtaining suitable auxiliary variables in many practical applications.

In the context of addressing label uncertainties within DCM, previous research has employed modified maximum likelihood estimators \citep{hausman1998misclassification,paleti2019misclassification,hausman2001mismeasured}. This approach involves the direct codification and estimation of the proportion of misclassified data using maximum likelihood estimation. Nevertheless, this method assumes that the proportion of misclassified data is fixed and can be applied to the out-of-sample data. In many scenarios, the extent of misclassification is not deterministic but falls within specific ranges.

To overcome the aforementioned limitations, we propose the application of robust optimization to handle feature and label uncertainties in DCM. Robust optimization offers a novel approach by accounting for data uncertainties within a predefined range. Unlike conventional econometric methods, robust optimization does not require prior knowledge of the error distribution or the collection of auxiliary data. Specifically, it assumes that uncertain parameters, such as measurement errors and the number of mislabeled choices, belong to an uncertainty set of possible outcomes. The optimization process is based on identifying and addressing the worst-case scenario within this uncertainty set \citep{gorissen2015practical,bertsimas2010optimality}. 

Robust optimization techniques have provided researchers with valuable tools to tackle problems involving data uncertainty across a wide array of domains, including transportation routing and scheduling \citep{shi2019robust,sungur2008robust,guo2024robust}, path recommendation \citep{mo2023robust}, healthcare resource allocation \citep{wang2019distributionally}, and portfolio optimization \citep{fernandes2016adaptive}.  Notably, to the best of our knowledge, no existing studies have employed robust optimization to tackle the challenges of measurement errors in travel behavior modeling with DCMs.

The robust optimization is solved under the worst-case scenario to provide a conservative decision. It has been shown in many studies that though the actual testing case is not the worst case, the decisions made by robust optimization still work well in general compared to the nominal model. In this paper, we derive the estimated parameters assuming some ``worst-case'' pattern of data errors that deteriorates the likelihood function. In the testing data, though the actual error patterns are not ``worst-case'', the estimated parameters may still perform better than the nominal DCM models. The intuition behind this can be understood using the analogy of robustness and regularization. Robustness can be seen as a way to avoid the model from overfitting. 

\subsection{Robust optimization for classification problems}
Robust optimization has emerged as a key methodology for tackling classification tasks in the presence of uncertainty. The theoretical groundwork for robust optimization was established by Ben-Tal and Nemirovski \citep{ben1998robust,ben1999robust}, who developed practical reformulations for optimization problems involving uncertainty. These foundational concepts were later applied to classification, particularly through the extension of support vector machines (SVMs) to more robust formulations. For example, \citet{xu2009robustness} introduced a robust SVM approach that accounts for input uncertainty using ellipsoidal uncertainty sets, showing that such models naturally incorporate regularization and can enhance generalization performance.

Building on these initial ideas, later studies explored robustness in both linear and nonlinear classification frameworks. \citet{shivaswamy2006second} introduced robust kernel-based classifiers, while \citet{lanckriet2002robust} proposed semidefinite programming formulations that enabled robustness in more complex, nonlinear decision boundaries. These early advancements laid the groundwork for viewing robust classification as a problem of optimizing performance under worst-case perturbations in the input space.

More recent developments have focused on distributionally robust optimization, which seeks to minimize classification loss under the most adverse distribution drawn from a specified ambiguity set. Key contributions in this area include \citet{duchi2021statistics}, who utilized the Wasserstein distance to define neighborhoods around empirical distributions, and \citet{duchi2019variance}, who introduced ambiguity sets based on chi-squared divergence to improve both generalization and fairness. Building on these foundations, \citet{sinha2017certifying} proposed scalable methods for training deep learning models within the DRO framework, emphasizing the critical role that the choice of divergence metric plays in shaping classifier performance.

Another influential strand of literature connects robust optimization with adversarial machine learning. \citet{goodfellow2014explaining} introduced adversarial examples, and revealed deep classifiers’ vulnerabilities to small perturbations. Building on this, \citet{madry2017towards} formulated adversarial training as a robust optimization problem, where the classifier is optimized to withstand worst-case perturbations within a norm-bounded set. TRADES \citep{zhang2019theoretically} refined this approach by introducing a trade-off between natural accuracy and robustness, grounded in a distributionally robust perspective.

Existing work on robust optimization for classification has predominantly focused on improving parameter estimation from training data with noisy features, often limited to binary classification settings. In contrast, we develop a unified framework that explicitly accounts for uncertainties in both features and labels, and extends beyond binary cases to include multinomial logit models. We derive closed-form robust counterparts for the BNL and MNL models—exact in some cases and accompanied by provable approximation bounds in others. Additionally, we analyze the statistical properties of the proposed estimators, showing that they yield higher Fisher information and exhibit improved out-of-sample generalization.

\section{Robustness against uncertainties in features}\label{sec_feature_uncertainty}

In this section, we consider perturbations (i.e., measurement errors) $\boldsymbol{\Delta x}_{n}$ for person $n$ in his/her features.
Without loss of generality, let us assume all alternatives use full features (i.e., $\bx_{n,i} = \bx_{n,j} = \bx_{n}$ for all $i,j \in \mathcal{C}_n$). For any model specification, we can set corresponding parameters to zero so as to filter out undesired features in an alternative. This is equivalent to defining a parameter domain $\boldsymbol{\beta}_i \in \mathcal{B}_i$, where:
\begin{align}
    \mathcal{B}_i = \{\boldsymbol{\beta}_i \in \mathbb{R}^{|\mathcal{K}|}:\; \beta_i^{(k)} = 0 \text{ if the $k$-th feature is not used for mode $i$} \}, \quad\forall i \in \mathcal{C}.
\end{align}
The overall parameter domain is thus as $\mathcal{B} = \cup_{i\in\mathcal{C}} \mathcal{B}_i$.
With uncertainty in features, we have 
\begin{align}
    U_{n,i} = \boldsymbol{\beta}_i^T (\boldsymbol{x}_{n}+\boldsymbol{\Delta x}_{n}) + \epsilon_{n,i},
\end{align}
where $\boldsymbol{\Delta x}_n \in \mathcal{Z}_n(\rho_n)$ and $\mathcal{Z}_n(\rho_n)= \{\boldsymbol{\Delta x}_n: \norm{\boldsymbol{\Delta x}_{n}}_{p} \leq \rho_{n}\}$ is the uncertainty set (we consider an $\ell_p$-norm uncertainty). $\rho_{n}$ is a hyper-parameter that represents the largest error of the perturbations $\boldsymbol{\Delta x}_n$.

Assume the uncertainties are independent across individuals:
\begin{align}
    \mathcal{Z}(\boldsymbol{\rho}) = \prod_{n\in\mathcal{N}} \mathcal{Z}_n(\rho_n) =  \prod_{n\in\mathcal{N}}\{\boldsymbol{\Delta x}_n: \norm{\boldsymbol{\Delta x}_{n}}_{p} \leq \rho_{n}\},
\end{align}
where $\boldsymbol{\rho}:=(\rho_n)_{n\in\mathcal{N}}$.

Then, the robust-feature MNL can be formulated as:
\begin{align}
    \max_{\boldsymbol{\beta}\in\mathcal{B}} \min_{ \boldsymbol{\Delta x}\in\mathcal{Z}(\boldsymbol{\rho})}  \sum_{n\in\mathcal{N}} \sum_{i\in\mathcal{C}_n} y_{n,i} \cdot \log\left(\frac{\exp({\boldsymbol{\beta}_i^T (\boldsymbol{x}_{n} + \boldsymbol{\Delta x}_{n})})}{\sum_{j\in\mathcal{C}_n} \exp({\boldsymbol{\beta}_j^T (\boldsymbol{x}_{n}+\boldsymbol{\Delta x}_{n}})}\right).
    \label{eq_robust_mnl}
\end{align}

The selection of $\boldsymbol{\rho}$ depends on the data patterns. In practice, we can follow the typical hyper-parameter tuning ideas in machine learning by reserving some of the training data as a validation data set or using cross-validation to select the best uncertainty sets.

\subsection{Binary logit model}\label{sec_binary}
The derivation of robust BNL is adapted from \citet{bertsimas2019robust}'s work on robust binary logistic regression (i.e., not the original work of the paper). This is because the BNL model is equivalent to logistic regression in the binary classification case. Consider a binary logit model (BNL) with at most two alternatives for each individual (i.e., $\mathcal{C} = \{1,2\}$). Define $I_n\in \mathcal{C}$ as the choice for individual $n$, and $J_n\in \mathcal{C}$ as the counterpart (i.e., non-choice), where $y_{n,I_n} = 1, y_{n,J_n} = 0,\; \forall n \in\mathcal{N}$.
Then the robust-feature BNL model can be reformulated as:
\begin{align}
  \max_{\boldsymbol{\beta}\in\mathcal{B}} \min_{\boldsymbol{\Delta x} \in \mathcal{Z}}\sum_{n\in\mathcal{N}} \log\left(\frac{1}{ 1 +\exp\big(-({\boldsymbol{\beta}_{I_n} -\boldsymbol{\beta}_{J_n})^T (\boldsymbol{x}_{n}} + \boldsymbol{\Delta x}_{n})\big)}\right).
\end{align}
The inner minimization problem is:
\begin{align}
 & \min_{\boldsymbol{\Delta x} \in \mathcal{Z}}\sum_{n\in\mathcal{N}} -\log\left({ 1 +\exp\big(-({\boldsymbol{\beta}_{I_n} -\boldsymbol{\beta}_{J_n})^T (\boldsymbol{x}_{n}} + \boldsymbol{\Delta x}_{n})\big)}\right) \\
   \Longleftrightarrow &  \sum_{n\in\mathcal{N}} \min_{\boldsymbol{\Delta x}_n \in \mathcal{Z}_n} -\log\left({ 1 +\exp\big(-({\boldsymbol{\beta}_{I_n} -\boldsymbol{\beta}_{J_n})^T (\boldsymbol{x}_{n}} + \boldsymbol{\Delta x}_{n})\big)}\right) .
  \label{eq_inner_binary}
\end{align}
Let $s_n = ({\boldsymbol{\beta}_{I_n} -\boldsymbol{\beta}_{J_n})^T (\boldsymbol{x}_{n}} + \boldsymbol{\Delta x}_{n})$, and define $g(s_n) = - \log(1+\exp(-s_n))$. Notice that $g(s_n)$ is strictly increasing with the increase in $s_n$. Hence, for each $n \in\mathcal{N}$, to minimize the objective function in Eq. \ref{eq_inner_binary}, we only need to minimize the following:
\begin{align}
\min_{\norm{\boldsymbol{\Delta x}_n}_p \leq \rho_n } -({\boldsymbol{\beta}_{I_n} -\boldsymbol{\beta}_{J_n})^T (\boldsymbol{x}_{n}} + \boldsymbol{\Delta x}_{n}) \quad \forall n \in \mathcal{N}.
  \label{eq_inner_binary_norm}
\end{align}

\begin{lemma}[Dual norm]
\emph{
\label{lemma_dual_norm}
Let $\boldsymbol{x}$ be a vector. $\norm{\boldsymbol{x}}_p$ is the $\ell_p$ norm of $\boldsymbol{x}$. Then, for any given vector $\boldsymbol{v}$, the dual norm problem is:
\begin{align}
    \max_{\norm{\boldsymbol{x}}_p \leq \rho} \{\boldsymbol{v}^T \boldsymbol{x}\} = \rho \cdot \norm{\boldsymbol{v}}_q, \qquad
    \min_{\norm{\boldsymbol{x}}_p \leq \rho} \{\boldsymbol{v}^T \boldsymbol{x}\} = -\rho \cdot \norm{\boldsymbol{v}}_q ,
\end{align}
where  $\norm{\cdot}_q$ is called the dual norm of $\norm{\cdot}_p$ and $\frac{1}{q}+\frac{1}{p} = 1$
}
\end{lemma}
Lemma \ref{lemma_dual_norm} is a direct result of Hölder's inequality \citep{holder1889ueber}. According to Lemma \ref{lemma_dual_norm}, the optimal objective function of Eq. \ref{eq_inner_binary_norm} is  
\begin{align}
   \min_{\norm{\boldsymbol{\Delta x}_n}_p \leq \rho_n } -({\boldsymbol{\beta}_{I_n} -\boldsymbol{\beta}_{J_n})^T (\boldsymbol{x}_{n}} + \boldsymbol{\Delta x}_{n}) =  -(\boldsymbol{\beta}_{I_n} -\boldsymbol{\beta}_{J_n})^T \boldsymbol{x}_{n} + \rho_n \norm{\boldsymbol{\beta}_{I_n} -\boldsymbol{\beta}_{J_n}}_q , \label{eq_dual_norm_max}
\end{align}
where $\frac{1}{q}+\frac{1}{p} = 1$. 

Substituting the optimal value into Eq. \ref{eq_inner_binary}, the robust binary logit model becomes:
\begin{align}
& \max_{\boldsymbol{\beta} \in \mathcal{B}}\sum_{n\in\mathcal{N}} -\log\left({ 1 +\exp\big(-({\boldsymbol{\beta}_{I_n} -\boldsymbol{\beta}_{J_n})^T\boldsymbol{x}_{n}} +\rho_n\norm{\boldsymbol{\beta}_{I_n} -\boldsymbol{\beta}_{J_n}}_q \big)}\right) \label{eq_bnl_logit} \\
\Longleftrightarrow &\max_{\boldsymbol{\beta} \in \mathcal{B}}\sum_{n\in\mathcal{N}} \log\left(\frac{\exp\big(\boldsymbol{\beta}_{I_n}^T \boldsymbol{x}_{n}\big)}{ \exp\big(\boldsymbol{\beta}_{I_n}^T \boldsymbol{x}_{n}\big)+\exp\big(\boldsymbol{\beta}_{J_n}^T\boldsymbol{x}_{n} +\rho_n\norm{\boldsymbol{\beta}_{I_n} -\boldsymbol{\beta}_{J_n}}_q \big)}\right) .
\label{eq_robust_bnl} 
\end{align}

\begin{remark}
Compared to the nominal BNL, the robust-feature counterpart of the BNL model has an additional $\rho_n \norm{\boldsymbol{\beta}_{I_n} -\boldsymbol{\beta}_{J_n}}_q$ term in the exponent of the logit function (Eq. \ref{eq_bnl_logit}). It resembles the $\ell_q$-regularization term in typical machine learning problems (such as logistic regression). However, the additional term from robustness penalizes model complexity in the log-odds ratio, whereas the typical regularization term is a linear penalty on the entire likelihood.
\end{remark}

To see the connections between the robust BNL and the typical regularization in machine learning, we can take the first-order Taylor approximation of Eq. \ref{eq_bnl_logit}. Define:
\begin{align}\label{eq_taylor1}
    h_n( z ) = -\log\left({ 1 +\exp\big(-({\boldsymbol{\beta}_{I_n} -\boldsymbol{\beta}_{J_n})^T\boldsymbol{x}_{n}} + z \big)}\right) .
\end{align}
Then the first-order Taylor approximation of $h_n(z)$ at $ z = 0$ is
\begin{align}
 h_n(z) \approx -\log\left({ 1 +\exp\big(-({\boldsymbol{\beta}_{I_n} -\boldsymbol{\beta}_{J_n})^T\boldsymbol{x}_{n}}  \big)}\right) - \frac{\exp\big(-({\boldsymbol{\beta}_{I_n} -\boldsymbol{\beta}_{J_n})^T\boldsymbol{x}_{n}}  \big)}{{ 1 +\exp\big(-({\boldsymbol{\beta}_{I_n} -\boldsymbol{\beta}_{J_n})^T\boldsymbol{x}_{n}}  \big)}} \cdot z .
\end{align}
Substitute $z = \rho_n\norm{\boldsymbol{\beta}_{I_n} -\boldsymbol{\beta}_{J_n}}_q$ we have:
\begin{align}\label{eq_taylor2}
  &\sum_{n\in\mathcal{N}}-\log\left({ 1 +\exp\big(-({\boldsymbol{\beta}_{I_n} -\boldsymbol{\beta}_{J_n})^T\boldsymbol{x}_{n}} + z \big)}\right) \notag \\
  &  \approx \sum_{n\in\mathcal{N}}\log\left(\frac{\exp\big(\boldsymbol{\beta}_{I_n}^T \boldsymbol{x}_{n}\big)}{ \exp\big(\boldsymbol{\beta}_{I_n}^T \boldsymbol{x}_{n}\big)+\exp\big(\boldsymbol{\beta}_{J_n}^T\boldsymbol{x}_{n} \big)}\right)  - \frac{\exp\big(({\boldsymbol{\beta}_{J_n} -\boldsymbol{\beta}_{I_n})^T\boldsymbol{x}_{n}}  \big)}{{ 1 +\exp\big(({\boldsymbol{\beta}_{J_n} -\boldsymbol{\beta}_{I_n})^T\boldsymbol{x}_{n}}  \big)}} \cdot \rho_n\norm{\boldsymbol{\beta}_{I_n} -\boldsymbol{\beta}_{J_n}}_q \notag \\
  & =\sum_{n\in\mathcal{N}} \log\left(P_{n,I_n}\right) - \sum_{n\in\mathcal{N}}\left(1 - P_{n,I_n}\right)\cdot \rho_n\norm{\boldsymbol{\beta}_{I_n} -\boldsymbol{\beta}_{J_n}}_q .
\end{align}
Therefore, when $\rho_n\norm{\boldsymbol{\beta}_{I_n} -\boldsymbol{\beta}_{J_n}}_q$ is small, the robust BNL is approximately equivalent to the $\ell_q$ regularization in machine learning problems with penalty coefficients as $\sum_{n\in\mathcal{N}}\left(1 - P_{n,I_n}\right)\cdot\rho_n$. This means that unlike typical machine learning regularization, this regularization effect will not diminish with increasing sample size.

\begin{remark}\label{remark_bnl_rho}
 When $\rho_n = 0$, the robust-feature BNL will fall back to the conventional BNL model. When $\rho_n = +\infty$ (which represents an extreme scenario where the data errors could go to infinity), the optimal value will be achieved when $\norm{\boldsymbol{\beta}_{I_n} -\boldsymbol{\beta}_{J_n}}_q = 0, \forall n\in\mathcal{N}$ (i.e., $\boldsymbol{\beta}_{I_n} =\boldsymbol{\beta}_{J_n}$). In DCM, a feature is actually only put in one alternative for estimation purposes (i.e., ${\beta}_{i}^{(k)} = 0$ or ${\beta}_{j}^{(k)} = 0$ for a feature $k\in\mathcal{K}$, $i,j\in\mathcal{C}$). Therefore, $\rho_n = +\infty$ will force the estimated $\boldsymbol{\beta}$ to be 0. 
\end{remark}

An extension of the robust BNL model is to consider a more general uncertainty set $\Tilde{\mathcal{Z}}$ with multiple norm constraints. Let the set of all norm constraints be $\mathcal{P}$, then
\begin{align}
    \Tilde{\mathcal{Z}} =  \prod_{n\in\mathcal{N}}\prod_{i\in\mathcal{P}}\{\boldsymbol{\Delta x}_n: \norm{\boldsymbol{\Delta x}_{n}}_{p_i} \leq \rho_{n}^{(i)}\} .
    \label{eq_general_uncertainty_set}
\end{align}
\begin{lemma}\label{lemma_dual_norm_extend}
[Dual norm with multiple constraints]
\emph{
Let $\boldsymbol{x}$ be a vector. $\norm{\boldsymbol{x}}_p$ is the $\ell_p$ norm of $\boldsymbol{x}$. Then, for any given vector $\boldsymbol{v}$, the dual norm problem with multiple constraints is:
\begin{align}
    \max_{\boldsymbol{x} \in \prod_{i\in\mathcal{P}}\{ \norm{\boldsymbol{x}}_{p_i} \leq \rho^{(i)}\}} \{\boldsymbol{v}^T \boldsymbol{x}\} =\min_{\boldsymbol{v}_i} \sum_{i \in \mathcal{P}}\rho^{(i)} \cdot \norm{\boldsymbol{v}_i}_{q_i} \quad \text{s.t.} \sum_{i\in\mathcal{P}}\boldsymbol{v}_i = \boldsymbol{v} , \\
    \min_{\boldsymbol{x} \in \prod_{i\in\mathcal{P}}\{ \norm{\boldsymbol{x}}_{p_i} \leq \rho^{(i)}\}} \{\boldsymbol{v}^T \boldsymbol{x}\} = \max_{\boldsymbol{v}_i} \sum_{i \in \mathcal{P}} -\rho^{(i)} \cdot \norm{\boldsymbol{v}_i}_{q_i} \quad \text{s.t.} \sum_{i\in\mathcal{P}}\boldsymbol{v}_i = \boldsymbol{v} ,
\end{align}
where  $\norm{\cdot}_{q_i}$ is called the dual norm of $\norm{\cdot}_{p_i}$ and $\frac{1}{q_i}+\frac{1}{p_i} = 1, \; \forall i \in \mathcal{P}$
}
\end{lemma}
Lemma \ref{lemma_dual_norm_extend} is a direct result of Lemma 9 in \citet{ben2015deriving}. Therefore, Eq. \ref{eq_inner_binary_norm} with multiple uncertainty constraints can be reformulated as
\begin{align}
   & \min_{\prod_{i\in\mathcal{P}}\{\boldsymbol{\Delta x}_n:\; \norm{\boldsymbol{\Delta x}_{n}}_{p_i} \leq \rho_{n}^{(i)}\} }  
   -({\boldsymbol{\beta}_{I_n} -\boldsymbol{\beta}_{J_n})^T (\boldsymbol{x}_{n}} + \boldsymbol{\Delta x}_{n}) \notag \\
   & = \max_{\boldsymbol{w}_n^{(i)}}-(\boldsymbol{\beta}_{I_n} -\boldsymbol{\beta}_{J_n})^T \boldsymbol{x}_{n} + \sum_{i\in\mathcal{P}}\rho_n^{(i)} \norm{\boldsymbol{w}_n^{(i)}}_{q_i}, \quad  \text{s.t.} \sum_{i\in\mathcal{P}} \boldsymbol{w}_n^{(i)} = \boldsymbol{\beta}_{I_n} -\boldsymbol{\beta}_{J_n} ,\label{eq_dual_norm_max_extend}
\end{align}
where $\frac{1}{q_i}+\frac{1}{p_i} = 1, \forall i \in\mathcal{P}$. 
The final robust binary logit model with multiple uncertainty constraints becomes:
\begin{align}
& \max_{\boldsymbol{\beta} \in \mathcal{B},\; \boldsymbol{w}}\sum_{n\in\mathcal{N}} \log\left(\frac{\exp\big(\boldsymbol{\beta}_{I_n}^T \boldsymbol{x}_{n}\big)}{ \exp\big(\boldsymbol{\beta}_{I_n}^T \boldsymbol{x}_{n}\big)+\exp\big(\boldsymbol{\beta}_{J_n}^T\boldsymbol{x}_{n} +\sum_{i\in\mathcal{P}}\rho_n^{(i)} \norm{\boldsymbol{w}_n^{(i)}}_{q_i} \big)}\right) \notag \\
& \text{s.t.} \; \sum_{i\in\mathcal{P}}\boldsymbol{w}_n^{(i)} = \boldsymbol{\beta}_{I_n} -\boldsymbol{\beta}_{J_n} \quad \forall n \in \mathcal{N} .
\label{eq_robust_bnl_extension}
\end{align}
Two widely used examples for multiple-norm uncertainty sets are 1) ball-box uncertainty set (i.e., $\{\boldsymbol{\Delta x}_n: \norm{\boldsymbol{\Delta x}_{n}}_{2} \leq \rho_{n}^{(2)}, \norm{\boldsymbol{\Delta x}_{n}}_{\infty} \leq \rho_{n}^{(\infty)}\}$) and 2) budget (box-polyhedral) uncertainty set (i.e., $\{\boldsymbol{\Delta x}_n: \norm{\boldsymbol{\Delta x}_{n}}_{1} \leq \rho_{n}^{(1)}, \norm{\boldsymbol{\Delta x}_{n}}_{\infty} \leq \rho_{n}^{(\infty)}\}$) \citep{bertsimas2022robustbook}.
By including multiple norm constraints with proper hyper-parameter tuning, we could better capture the potential error patterns, and prevent the robust model from choosing unreasonable points as the worst-case scenario.

\subsection{Multinomial logit model}

The robustification of the multinomial logit model (MNL) is more difficult than the binary model. The inner maximization problem of the robust MNL is equivalent to ``maximizing a convex function''. For the robust BNL model, the convex function is monotonically increasing, which leads to a direct simplification of the problem. However, the MNL model cannot be simplified in a similar way. In this study, we approximate the robust MNL problem using Jensen's inequality, which results in a similar formulation as the robust BNL model. 

Similarly, let $I_n\in \mathcal{C}_n$ be the choice index for individual $n$. Replacing the objective function of Eq. \ref{eq_robust_mnl} by $t$, we have:
\begin{subequations}\label{eq_robust_mnl_with_t}
 \begin{align}
    \max_{\boldsymbol{\beta} \in \mathcal{B},\; t} & \quad t \\
     \text{s.t.} & \min_{\boldsymbol{\Delta x} \in \mathcal{Z}} \sum_{n\in\mathcal{N}} \left[-\log\left(\sum_{j\in\mathcal{C}_n}\exp\big(({\boldsymbol{\beta}_j-\boldsymbol{\beta}_{I_n})^T (\boldsymbol{x}_{n}+\boldsymbol{\Delta x}_{n})\big)}\right)\right] \geq t .
\end{align}   
\end{subequations}
Since the uncertainty sets are independent across individuals, similar to Eq. \ref{eq_inner_binary}, we have
\begin{align}
   & \min_{\boldsymbol{\Delta x} \in \mathcal{Z}} \sum_{n\in\mathcal{N}}-\log\left(\sum_{j\in\mathcal{C}_n}\exp\big(({\boldsymbol{\beta}_j-\boldsymbol{\beta}_{I_n})^T (\boldsymbol{x}_{n}+\boldsymbol{\Delta x}_{n})\big)}\right) \geq t \notag \\
   \Longleftrightarrow & \sum_{n\in\mathcal{N}} \min_{\substack{\boldsymbol{\Delta x}_n \in \mathcal{Z}_n }} -\log\left(\sum_{j\in\mathcal{C}_n}\exp\big(({\boldsymbol{\beta}_j-\boldsymbol{\beta}_{I_n})^T (\boldsymbol{x}_{n}+\boldsymbol{\Delta x}_{n})\big)}\right) \geq t .\label{eq_mnl_robust_const}
\end{align}
Eq. \ref{eq_mnl_robust_const} can be reformulated as:
\begin{align}
 \sum_{n\in\mathcal{N}}\max_{\substack{\boldsymbol{\Delta x}_n \in \mathcal{Z}_n }} \log\left(\sum_{j\in\mathcal{C}_n}\exp \big((\boldsymbol{\beta}_j-\boldsymbol{\beta}_{I_n})^T (\boldsymbol{x}_{n}+\boldsymbol{\Delta x}_{n})\big)\right) \leq -t .\label{eq_inner_min_ind}
\end{align}
Note that we eliminate the negative sign and change the formulation to a maximization problem.

Consider the inner maximization problem in Eq. \ref{eq_inner_min_ind}, according to the Jensen's inequality, we have:
\begin{align}\label{eq_jensen_ineq}
 \max_{\substack{\boldsymbol{\Delta x}_n \in \mathcal{Z}_n }} \log\left(\sum_{j\in\mathcal{C}_n}\exp \big((\boldsymbol{\beta}_j-\boldsymbol{\beta}_{I_n})^T (\boldsymbol{x}_{n}+\boldsymbol{\Delta x}_{n})\big)\right)  
 \leq
  \log\left(\sum_{j\in\mathcal{C}_n}\exp \big(\max_{\substack{\boldsymbol{\Delta x}_n \in \mathcal{Z}_n }}(\boldsymbol{\beta}_j-\boldsymbol{\beta}_{I_n})^T (\boldsymbol{x}_{n}+\boldsymbol{\Delta x}_{n})\big)\right) 
\end{align}
With the same derivation as Eq. \ref{eq_dual_norm_max}, we now have:
\begin{align}
    \max_{\norm{\boldsymbol{\Delta x}_n}_p \leq \rho_n}(\boldsymbol{\beta}_j-\boldsymbol{\beta}_{I_n})^T (\boldsymbol{x}_{n}+\boldsymbol{\Delta x}_{n}) = (\boldsymbol{\beta}_{j} -\boldsymbol{\beta}_{I_n})^T \boldsymbol{x}_{n} + \rho_n \norm{\boldsymbol{\beta}_{j} -\boldsymbol{\beta}_{I_n}}_q,
\end{align}
where $\frac{1}{q}+\frac{1}{p} = 1$. Therefore, Eq. \ref{eq_inner_min_ind} can be approximated as:
\begin{align}
\label{eq_jesen_const}
\sum_{n\in\mathcal{N}}\log\left(\sum_{j\in\mathcal{C}_n}\exp\bigg((\boldsymbol{\beta}_{j} -\boldsymbol{\beta}_{I_n})^T \boldsymbol{x}_{n} + \rho_n \norm{\boldsymbol{\beta}_{j} -\boldsymbol{\beta}_{I_n}}_q\bigg)\right)
 \leq -t .
\end{align}
And the approximation of the robust-feature DCM can be reformulated as:
\begin{align}\label{eq_robust_mnl_jesen}
    \max_{\boldsymbol{\beta}\in \mathcal{B}} & \sum_{n\in\mathcal{N}}\log\left(\frac{\exp(\boldsymbol{\beta}_{I_n}^T \boldsymbol{x}_{n}) }{\sum_{j\in\mathcal{C}_n}\exp(\boldsymbol{\beta}_j^T \boldsymbol{x}_{n} + \rho_{n} \norm{\bb_j - \bb_{I_n}}_q )}\right) .
\end{align}
It has a similar form as the robust BNL model (Eq. \ref{eq_robust_bnl}). Actually, when $\mathcal{C}=2$, the robust MNL problem will reduce to the robust BNL problem. However, robust BNL is an exact robust counterpart while robust MNL is an approximation.  

\begin{remark}
The solution of Eq. \ref{eq_robust_mnl_jesen} is a lower bound of the original robust MNL problem (Eq. \ref{eq_robust_mnl}). The reason is that, Constraint \ref{eq_jesen_const} is more restricted than the original constraint (Eq. \ref{eq_inner_min_ind}), which gives a smaller feasible region. Therefore, the optimal objective function in Eq. \ref{eq_robust_mnl_jesen} is smaller than that in Eq. \ref{eq_robust_mnl} (i.e., lower bound for a maximization problem). The inequality in Eq. \ref{eq_jensen_ineq} is tight when 1) $(\boldsymbol{\beta}_j-\boldsymbol{\beta}_{I_n})=(\boldsymbol{\beta}_i-\boldsymbol{\beta}_{I_n})$ for all $i,j\in\mathcal{C}_n$. 2) Or one $\bb_i - \bb_{I_n}$ significantly larger than others. (i.e., $\exists i^*$ such that $(\bb_{i^*} - \bb_{I_n})^T(\boldsymbol{x}_{n}+\boldsymbol{\Delta x}_{n}) \gg (\bb_{j} - \bb_{I_n})^T(\boldsymbol{x}_{n}+\boldsymbol{\Delta x}_{n}), j\neq i^*$. The first case implies $\bb = 0$, which will be achieved when $\rho_n \rightarrow \infty$. More analysis on the tightness of the inequality can be found in \ref{append_ineq_gap}. Moreover, we also derive an upper bound of the original robust MNL problem (\ref{append_upper}), combining with the lower bound could provide the optimality gap for the approximation. 
\end{remark}

Similar to the extension of Robust BNL, for a general uncertainty set $\Tilde{\mathcal{Z}}$ (Eq. \ref{eq_general_uncertainty_set}), the robust MNL problem is:
\begin{align}\label{eq_robust_mnl_jesen_extension}
&    \max_{\boldsymbol{\beta}\in \mathcal{B},\; \boldsymbol{w}} \sum_{n\in\mathcal{N}}\log\left(\frac{\exp(\boldsymbol{\beta}_{I_n}^T \boldsymbol{x}_{n}) }{\sum_{j\in\mathcal{C}_n}\exp(\boldsymbol{\beta}_j^T \boldsymbol{x}_{n} + \sum_{i\in\mathcal{P}} \rho_{n}^{(i)} \norm{\boldsymbol{w}_n^{(i, j)}}_q )}\right) \notag \\
& \text{s.t.} \; \sum_{i\in\mathcal{P}}\boldsymbol{w}_n^{(i, j)} = \boldsymbol{\beta}_{j} -\boldsymbol{\beta}_{I_n} \quad \forall n \in \mathcal{N}, \; \forall j\in\mathcal{C}_n .
\end{align}

% Another way to solve the robust MNL is to approximate the original problem to a robust linear optimization over a convex set using conjugate function and outer approximation. We describe the detailed methodology in \ref{sec_method2}. However, due to its difficulties in implementation, we do not apply it to the case study.

It is worth noting that the robust-feature MNL model is solved under the worst-case scenario. Whether it works well or not depends on how the uncertainty set is defined and what are the error patterns in the testing data. To avoid solutions from being too conservative, we need to choose a proper value of $\rho_n$, or integrate multiple norm constraints to define the uncertainty set.

\section{Robustness against uncertainties in labels}\label{sec_label}
Section \ref{sec_feature_uncertainty} discusses the possible uncertainties in features (i.e., $\boldsymbol{\Delta x}$). The measurement errors may also apply to labels or individual choices (i.e., $\boldsymbol{\Delta y}$). In this study, we consider the following uncertainty set:
\begin{align}
    \mathcal{U}(\Gamma) = \{\boldsymbol{\Delta y}:&\sum_{j\in\mathcal{C}_n} \Delta y_{n,j} = 0,\; \Delta y_{n,I_n} \in \{0,-1\},\forall n \in \mathcal{N}; \notag\\
    &\Delta y_{n,j} \in \{0,1\}\; \forall j\in\mathcal{C}_n\setminus\{I_n\},\; \forall n \in \mathcal{N}; \notag \\
    &\sum_{{n \in \mathcal{N}}} -\Delta y_{n,I_n} \leq \Gamma \} .
\end{align}
Specially, if $\Delta y_{n,I_n} = -1$ and $\Delta y_{n,j} = 1$, it means that the individual's actual choice is $I_n\in\mathcal{C}_n$ but the data mislabeled it as $j\in\mathcal{C}_n\setminus \{I_n\}$. The uncertainty set $ \mathcal{U}$ represents that there are at most $\Gamma$ mislabeled samples. Then the robust DCM problem against feature uncertainty can be represented as:
\begin{align}\label{eq_original_rl}
    \max_{\boldsymbol{\beta}\in \mathcal{B}} & \min_{\boldsymbol{\Delta y} \in \mathcal{U}(\Gamma)}\sum_{n\in\mathcal{N}} \sum_{i\in\mathcal{C}_n} (y_{n,i}+\Delta y_{n,i}) \cdot \log(P_{n,i}(\boldsymbol{\beta})) .
\end{align}
The formulation is general for both BNL and MNL cases.
% Replace the objective function by $t$, we have:
% \begin{align}
%     \max_{\boldsymbol{\beta}} & \quad t\\
%     \text{s.t. } & \min_{\boldsymbol{\Delta y} \in \mathcal{U}}\sum_{n\in\mathcal{N}} \sum_{i\in\mathcal{C}_n} (y_{n,i}+\Delta y_{n,i}) \cdot \log(P_{n,i}(\boldsymbol{\beta})) \geq t \label{eq_rb_label_cons}
% \end{align}
Consider the convex hull of $\mathcal{U}(\Gamma)$:
\begin{align}
        \text{Conv}\left(\mathcal{U}(\Gamma)\right) = \{\boldsymbol{\Delta y}:&\sum_{j\in\mathcal{C}_n} \Delta y_{n,j} = 0, -1 \leq \Delta y_{n,I_n}\leq 0,\forall n \in \mathcal{N}; \notag\\
    & 0 \leq \Delta y_{n,j} \leq 1\; \forall j\in\mathcal{C}_n\setminus\{I_n\},\; \forall n \in \mathcal{N}; \notag \\
    &\sum_{{n \in \mathcal{N}}} -\Delta y_{n,I_n} \leq \Gamma \} .
\end{align}
Because the inner minimization problem is linear in $\boldsymbol{\Delta y}$. And the extreme points for $\text{Conv}(\mathcal{U}\left(\Gamma)\right)$ are integers. Hence, the original inner minimization problem on $\mathcal{U}(\Gamma)$ is equivalent to minimizing over its convex hull:
\begin{subequations}\label{eq_opt_conv}
\begin{align}
    \min_{\boldsymbol{\Delta y}}&\sum_{n\in\mathcal{N}} \sum_{i\in\mathcal{C}_n} (y_{n,i}+\Delta y_{n,i}) \cdot \log(P_{n,i}(\boldsymbol{\beta})) \\
    \text{s.t. } & \Delta y_{n,I_{n}}+\sum_{j\in\mathcal{C}_n\setminus\{I_n\}} \Delta y_{n,j} = 0, \quad \forall n \in \mathcal{N}, \label{eq_const_convhull1}\\
    &0 \leq - \Delta y_{n,I_{n}} \leq 1, \quad \forall n \in \mathcal{N},\\
    &0 \leq \Delta y_{n,j} \leq 1, \quad  \forall n \in \mathcal{N}, \forall j\in\mathcal{C}_n\setminus\{I_n\},\label{eq_const_convhull2}\\
    &\sum_{{n \in \mathcal{N}}} -\Delta y_{n,I_n} \leq \Gamma.
\end{align}
\end{subequations}
As $\boldsymbol{\Delta y}$ is bounded, the optimal solution of Eq. \ref{eq_opt_conv} is also bounded for a given $\bb$. By strong duality, the optimal solution in Eq. \ref{eq_opt_conv} equals that of its dual problem
\begin{subequations}
\begin{align}
    \max_{\boldsymbol{\eta},\boldsymbol{\gamma},C}&\sum_{n\in\mathcal{N}} \sum_{i\in\mathcal{C}_n} y_{n,i} \cdot \log(P_{n,i}(\boldsymbol{\beta}))+\sum_{n\in\mathcal{N}} \sum_{i\in\mathcal{C}_n} \eta_{n,i}+\Gamma \cdot C \\
    \text{s.t. } & -\gamma_{n}+\eta_{n,I_{n}}+ C \leq \log(P_{n,I_{n}}(\boldsymbol{\beta})), \quad \forall n \in \mathcal{N},\label{eq_const1}\\
    & \gamma_{n}+\eta_{n,j} \leq \log(P_{n,j}(\boldsymbol{\beta})), \quad \forall n \in \mathcal{N}, \forall j\in\mathcal{C}_n\setminus\{I_n\},\label{eq_const3}\\
    &\eta_{n,i} \leq 0, \quad \forall n \in \mathcal{N}, i \in \mathcal{C}_n,\label{eq_const4}\\
    &C \leq 0,\label{eq_const2}
\end{align}
\end{subequations}
where $\boldsymbol{\gamma} = (\gamma_n)_{n\in\mathcal{N}}$, $\boldsymbol{\eta} = (\eta_{n,i})_{n\in\mathcal{N}, i\in\mathcal{C}_n}$, and $C$ are dual decision variables. 

The final formulation is just a combination of the inner and outer problems:
% \begin{align}
%     \max_{\boldsymbol{\beta}, \boldsymbol{u},m}&\sum_{n\in\mathcal{N}} \sum_{i\in\mathcal{C}_n} y_{n,i} \cdot \log(P_{n,i}(\boldsymbol{\beta}))+\sum_{n\in\mathcal{N}} \sum_{i\in\mathcal{C}_n}U_{n,i}+\Gamma m \\
%     \text{s.t. } & -\theta_{n}+u_{n,i_{n}}+m \leq \log(P_{ni_{n}}(\boldsymbol{\beta})), \forall n \in \mathcal{N}\label{eq_const1}\\
%     & \theta_{n}+u_{n,j} \leq \log(P_{n,j}(\boldsymbol{\beta})), \forall n \in \mathcal{N},\forall j\in\mathcal{C}_n\setminus\{I_n\}\\
%     &U_{n,i} \leq 0\\
%     &m \leq 0\label{eq_const2}
% \end{align}
\begin{subequations}\label{eq_robust_label_mnl}
\begin{align}
    \max_{\boldsymbol{\beta}\in\mathcal{B}, \boldsymbol{\eta},\boldsymbol{\gamma},C }&\sum_{n\in\mathcal{N}}  \log\left(\frac{\exp(\boldsymbol{\beta}_{I_n}^T \boldsymbol{x}_{n}) }{\sum_{j\in\mathcal{C}_n}\exp(\boldsymbol{\beta}_j^T \boldsymbol{x}_{n} )}\right)+\sum_{n\in\mathcal{N}} \sum_{i\in\mathcal{C}_n} \eta_{n,i}+\Gamma \cdot C \\
    \text{s.t. } & \text{Constraints } \ref{eq_const1} \sim \ref{eq_const2} .
\end{align}
\end{subequations}
This problem has a twice continuously differentiable concave objective function and constraints, making it tractably solvable with an interior point method.

\begin{remark}\label{remark_robust_label_gamma}
When $\Gamma = 0$, the optimal solution for $\boldsymbol{\eta}$ is $0$ because $\boldsymbol{\gamma}$ and $C$ become free variables. Then constraints \ref{eq_const1} and \ref{eq_const3} do not restrict $\boldsymbol{\eta}$ to take its maximum value. Therefore, when $\Gamma = 0$ (i.e., no uncertainty), Eq. \ref{eq_robust_label_mnl} is equivalent to the nominal MNL model, which validates the formulation. 

\end{remark}

\section{Statistical properties of the robust DCM estimators}\label{sec_stats}
In the context of econometrics, it is often of interest to understand the statistical properties of an estimator (i.e., consistency and efficiency). In this study, we first show that robust formulations (both robust feature and robust label) tend to shrink the scale of the estimated $\bb$ compared to nominal formulations, which implies that they are biased and inconsistent estimators. However, we also show that this shrinkage reduces variances and could make the model perform better in out-of-sample predictions, especially when there are errors in the testing data.  

\subsection{Statistical properties of the robust feature estimators}
It is worth noting that the robust BNL formulation (Eq. \ref{eq_robust_bnl}) is a special case of the robust MNL formulation (Eq. \ref{eq_robust_mnl_jesen}). Therefore, we only need to discuss the properties of robust MNL formulations.

\begin{prop}\label{prop_inconsist}

Let $\hat{\boldsymbol{\beta}}^{\text{RF}}$ be the solution of Eq. \ref{eq_robust_mnl_jesen} (robust feature MNL problem). $\hat{\boldsymbol{\beta}}^{\text{RF}}$ is an inconsistent estimate of the actual parameter. 
\end{prop}
\begin{proof}

For simplicity of description, we first consider a single sample $n$. Let us rewrite the objective function of Eq. \ref{eq_robust_mnl_jesen} for sample $n$ as 
\begin{align}\label{eq_ll_rf}
    \mathcal{LL}^{\text{RF}}_n = - \log \sum_{j\in\mathcal{C}_n}\exp(\left(\boldsymbol{\beta}_j - \bb_{I_n}\right)^T \boldsymbol{x}_{n} + \rho_{n} \norm{\bb_j - \bb_{I_n}}_q ).
\end{align}
Without loss of generality, let us assume all samples' choices are $I_n = I\in\mathcal{C}$, and define $\tilde{\bb}_j = (\bb_j - \bb_I)$. Then Eq. \ref{eq_ll_rf} can be transformed to:
\begin{align}\label{eq_ll_rf2}
    \mathcal{LL}^{\text{RF}}_n(\tilde{\bb}) = - \log \sum_{j\in\mathcal{C}_n}\exp\left(\tilde{\bb}_j^T \boldsymbol{x}_{n} + \rho_{n} \norm{\tilde{\bb}_j}_q \right).
\end{align}
Then maximizing Eq. \ref{eq_ll_rf2} is essentially minimizing
$\log \sum_{j\in\mathcal{C}_n}\exp\left(\tilde{\bb}_j^T \boldsymbol{x}_{n}\right) \cdot \exp\left(\rho_{n} \norm{\tilde{\bb}_j}_q \right)$. Compared to the nominal MNL model, there is an additional bias term $\exp(\rho_{n} \norm{\tilde{\bb}_j}_q)$ for every sample $n$, and it does not diminish as $|\mathcal{N}|$ goes to infinity. Hence, the robust-feature MNL estimator is inconsistent. And by definition, it is also inefficient. 

\end{proof}

\begin{remark}

Note that the value of $\hat{\boldsymbol{\beta}}^{\text{RF}}$ depends on external hyper-parameter $\rho_n$. Let $\hat{\tilde{\bb}}^{\text{RF}}$ be the transformation of $\hat{{\bb}}^{\text{RF}}$ as above. As $\rho_n$ goes to infinity, the optimal solution will be achieved at $\hat{\tilde{\bb}}^{\text{RF}} = 0$, which implies $\hat{{\bb}}^{\text{RF}}_j = \hat{{\bb}}^{\text{RF}}_I$ for all $j\in\mathcal{C}$. In the specification of DCM, one of the alternatives (i.e., the base alternative) will have a fixed coefficient of feature $k\in\mathcal{K}$ to be zero (to avoid perfect co-linearity). Therefore, $\exists i \in \mathcal{C}$ such that $\hat{{\bb}}_{i}^{\text{RF},(k)} = 0, \; \forall k \in\mathcal{K}$. Then $\hat{{\bb}}^{\text{RF}}_j = \hat{{\bb}}^{\text{RF}}_I$ is equivalent to $\hat{{\bb}}^{\text{RF}} = 0$. Therefore, a larger value of $\rho_n$ will shrink $\hat{\boldsymbol{\beta}}^{\text{RF}}$ toward 0.  

\end{remark}

\begin{remark}

We may also consider the case when the size of the uncertainty set goes to zero (i.e., $\rho_n \rightarrow 0$). In this case, the bias term has $\lim_{\rho_n \rightarrow 0} \exp(\rho_{n} \norm{\tilde{\bb}_j}_q) = 1$ because the function is continuous in $\rho_n$. Then the objective function becomes $\log \sum_{j\in\mathcal{C}_n}\exp\left(\tilde{\bb}_j^T \boldsymbol{x}_{n}\right)$, which reduces to the nominal MNL model. Therefore, when the size of the uncertainty set goes to zero, the robust formulation will converge to the consistent and unbiased MLE estimator. This is similar to the case where the distributional robust optimization reduces to the empirical risk minimization or stochastic optimization as the ambiguity set shrinks toward zero \citep{sun2016convergence,mohajerin2018data}.

\end{remark}

However, just like how regularization works in machine learning, a biased estimator may perform better in out-of-sample predictions due to lower estimated variances. Rigorously proving the relationship of variance between two estimators is difficult. In Proposition \ref{prop_info_matrix}, we show that the robust estimator yields a larger trace of the Fisher information matrix than the original MLE estimator, which in general could imply lower variances.  

\begin{prop}\label{prop_info_matrix}
Let $\mathcal{I}_{\text{RF}}(\cdot)$ and $\mathcal{I}_{\text{MLE}}(\cdot)$ be the Fisher information matrix of the robust-feature MNL estimator and the original MLE estimator, respectively. If $\rho_n$ is large enough for all $n\in\mathcal{N}$, we have
\begin{align}
    \text{Tr}\left(\mathcal{I}_{\text{RF}}(\boldsymbol{\hat{\boldsymbol{\beta}}^{\text{RF}}})\right) \geq     \text{Tr}\left(\mathcal{I}_{\text{MLE}}(\boldsymbol{\hat{\boldsymbol{\beta}}^{\text{MLE}}})\right).
\end{align}
where $\boldsymbol{\hat{\boldsymbol{\beta}}^{\text{RF}}}$ and $\boldsymbol{\hat{\boldsymbol{\beta}}^{\text{MLE}}}$ are corresponding estimated parameters. $\text{Tr}(\cdot)$ is the trace of a matrix. The specific conditions for $\rho_n$ are provided in the proof.
\end{prop}
\begin{proof}
Using the same transformation of $\bb$ as in the proof of Proposition \ref{prop_inconsist}, we define $\hat{\tilde{\boldsymbol{\beta}}}^{\text{RF}}$ and $\hat{\tilde{\boldsymbol{\beta}}}^{\text{MLE}}$, where  $\hat{\tilde{\boldsymbol{\beta}}}^{\text{RF}}_j = \hat{{\boldsymbol{\beta}}}^{\text{RF}}_j - \hat{{\boldsymbol{\beta}}}^{\text{RF}}_I$ and $\hat{\tilde{\boldsymbol{\beta}}}^{\text{MLE}}_j = \hat{{\boldsymbol{\beta}}}^{\text{MLE}}_j - \hat{{\boldsymbol{\beta}}}^{\text{MLE}}_I$ (for all $j\in\mathcal{C}$). Notice that this transformation is an orthogonal linear mapping: $\hat{\tilde{\boldsymbol{\beta}}}^{\text{RF}} = \boldsymbol{A}\cdot \hat{\boldsymbol{\beta}}^{\text{RF}}$, where $\boldsymbol{A}$ is an orthogonal matrix and $ \boldsymbol{A}^T\cdot \boldsymbol{A} = \boldsymbol{I}$. Therefore, to show $\text{Tr}\left(\mathcal{I}_{\text{RF}}(\boldsymbol{\hat{\boldsymbol{\beta}}^{\text{RF}}})\right) \geq     \text{Tr}\left(\mathcal{I}_{\text{MLE}}(\boldsymbol{\hat{\boldsymbol{\beta}}^{\text{MLE}}})\right)$, we only need to show  $\text{Tr}\left(\mathcal{I}_{\text{RF}}(\hat{\tilde{\boldsymbol{\beta}}}^{\text{RF}})\right) \geq     \text{Tr}\left(\mathcal{I}_{\text{MLE}}(\hat{\tilde{\boldsymbol{\beta}}}^{\text{MLE}})\right)$ because the Fisher information matrix is essentially the negative Hessian matrix, and orthogonal linear mapping does not change the eigenvalues of the Hessian matrix, thus the trace (sum of eigenvalues) are still the same.

Similar to Eq. \ref{eq_ll_rf}, let $\mathcal{LL}^{\text{MLE}}_n$ be the log-likelihood of the original MNL model for sample $n$, we can also transfer it to  
\begin{align}\label{eq_ll_mnl}
    \mathcal{LL}^{\text{MLE}}_n(\tilde{\bb}) = - \log \sum_{j\in\mathcal{C}_n}\exp\left(\tilde{\bb}_j^T \boldsymbol{x}_{n}\right).
\end{align}
Calculating the first derivatives of Eqs. \ref{eq_ll_rf} and \ref{eq_ll_mnl} gives:
\begin{align}
    & \frac{\partial \mathcal{LL}^{\text{RF}}_n}{\partial \tilde{\bb}_i} = -\frac{\exp\left(\tilde{\bb}_i^T \boldsymbol{x}_{n} + \rho_{n} \norm{\tilde{\bb}_i}_q\right)\cdot \left(\boldsymbol{x}_{n} + \rho_{n}\cdot \nabla_{\tilde{\bb}_i}\norm{\tilde{\bb}_i}_q  \right) }{ \sum_{j\in\mathcal{C}_n}\exp\left(\tilde{\bb}_j^T \boldsymbol{x}_{n} + \rho_{n} \norm{\tilde{\bb}_j}_q\right)} = -\tilde{P}_{i,n}^{\text{Norm}}\cdot \left(\boldsymbol{x}_n + \rho_{n}\cdot \nabla_{\tilde{\bb}_i}\norm{\tilde{\bb}_i}_q \right), \\
    & \frac{\partial \mathcal{LL}^{\text{MLE}}_n}{\partial \tilde{\bb}_i} = -\frac{\exp\left(\tilde{\bb}_i^T \boldsymbol{x}_{n}\right)\cdot \boldsymbol{x}_{n}}{ \sum_{j\in\mathcal{C}_n}\exp\left(\tilde{\bb}_j^T \boldsymbol{x}_{n}\right)} = -\tilde{P}_{i,n}\cdot \boldsymbol{x}_{n},
\end{align}
where $\tilde{P}_{i,n}^{\text{Norm}} := \frac{\exp\left(\tilde{\bb}_i^T \boldsymbol{x}_{n} + \rho_{n} \norm{\tilde{\bb}_i}_q\right)}{\sum_{j\in\mathcal{C}_n}\exp\left(\tilde{\bb}_j^T \boldsymbol{x}_{n} + \rho_{n} \norm{\tilde{\bb}_j}_q\right)}$ and $\tilde{P}_{i,n} := \frac{\exp\left(\tilde{\bb}_i^T \boldsymbol{x}_{n} \right)}{\sum_{j\in\mathcal{C}_n}\exp\left(\tilde{\bb}_j^T \boldsymbol{x}_{n} \right)}$ are auxiliary variables. For the second derivatives, we first consider the diagonal terms: 
\begin{align}\label{eq_info_rf_ik}
  \frac{\partial^2 \mathcal{LL}^{\text{RF}}_n}{\partial (\tilde{\bb}_i)^2} =&  -\tilde{P}_{i,n}^{\text{Norm}}\cdot(1-\tilde{P}_{i,n}^{\text{Norm}})\cdot\left(\boldsymbol{x}_n + \rho_{n}\cdot \nabla_{\tilde{\bb}_i}\norm{\tilde{\bb}_i}_q \right)\cdot \left(\boldsymbol{x}_n + \rho_{n}\cdot \nabla_{\tilde{\bb}_i}\norm{\tilde{\bb}_i}_q \right)^T  \notag \\
   & - \tilde{P}_{i,n}^{\text{Norm}}\cdot\rho_{n}\cdot\nabla^2_{\tilde{\bb}_i}\norm{\tilde{\bb}_i}_q ,\\
  \frac{\partial^2 \mathcal{LL}^{\text{MLE}}_n}{\partial (\tilde{\bb}_i)^2} =&  -\tilde{P}_{i,n}\cdot(1-\tilde{P}_{i,n})\cdot\boldsymbol{x}_n\cdot \boldsymbol{x}_n^T .
\end{align}
For off-diagonal terms:
\begin{align}
   \frac{\partial^2 \mathcal{LL}^{\text{RF}}_n}{\partial \tilde{\bb}_i\partial \tilde{\bb}_j} =&  \tilde{P}_{i,n}^{\text{Norm}}\cdot\tilde{P}_{j,n}^{\text{Norm}}\cdot\left(\boldsymbol{x}_n + \rho_{n}\cdot \nabla_{\tilde{\bb}_j}\norm{\tilde{\bb}_j}_q \right)\cdot \left(\boldsymbol{x}_n + \rho_{n}\cdot \nabla_{\tilde{\bb}_i}\norm{\tilde{\bb}_i}_q \right)^T  , \\
   \frac{\partial^2 \mathcal{LL}^{\text{MLE}}_n}{\partial \tilde{\bb}_i\partial \tilde{\bb}_j} =&  \tilde{P}_{i,n}\cdot\tilde{P}_{j,n}\cdot\boldsymbol{x}_n\cdot \boldsymbol{x}_n^T .
\end{align}
Consider the $k$-th feature in alternative $i$. Let of corresponding diagonal term in $\mathcal{I}_{\text{RF}}(\boldsymbol{\hat{\boldsymbol{\beta}}^{\text{RF}}})$ be $\mathcal{I}_{\text{RF}}(\boldsymbol{\hat{\boldsymbol{\beta}}^{\text{RF}}})_{(i,k),(i,k)} \in \mathbb{R}$, we have:
\begin{align}
    &\mathcal{I}_{\text{RF}}(\hat{\tilde{\boldsymbol{\beta}}}^{\text{RF}})_{(i,k),(i,k)} = \sum_{n\in\mathcal{N}}-\left(\frac{\partial^2 \mathcal{LL}^{\text{RF}}_n}{\partial (\tilde{\bb}_i)^2}\right)_{k,k} \notag \\
    &= \sum_{n\in\mathcal{N}}\tilde{P}_{i,n}^{\text{Norm}}\cdot(1-\tilde{P}_{i,n}^{\text{Norm}})\cdot\left(\left(\boldsymbol{x}_n^{(k)}\right)^2 + 2\rho_{n}\cdot \left(\boldsymbol{x}_n^{(k)}\right)\cdot\left(\nabla_{\tilde{\bb}_i}\norm{\tilde{\bb}_i}_q\right)_k
    + (\rho_{n})^2\cdot \left(\nabla_{\tilde{\bb}_i}\norm{\tilde{\bb}_i}_q\right)_k^2  \right) \notag \\
   & + \tilde{P}_{i,n}^{\text{Norm}}\cdot\rho_{n}\cdot\left(\nabla^2_{\tilde{\bb}_i}\norm{\tilde{\bb}_i}_q \right)_{k,k} .
\end{align}
Similarly, for the MLE estimator:
\begin{align}
    \mathcal{I}_{\text{MLE}}(\hat{\tilde{\boldsymbol{\beta}}}^{\text{MLE}})_{(i,k),(i,k)} = \sum_{n\in\mathcal{N}}-\left(\frac{\partial^2 \mathcal{LL}^{\text{MLE}}_n}{\partial (\tilde{\bb}_i)^2}\right)_{k,k} 
    = \sum_{n\in\mathcal{N}}\tilde{P}_{i,n}\cdot(1-\tilde{P}_{i,n})\cdot \left(\boldsymbol{x}_n^{(k)}\right)^2 .
\end{align}
Therefore, 
\begin{align}
    \text{Tr}(\mathcal{I}_{\text{RF}}(\hat{\tilde{\boldsymbol{\beta}}}^{\text{RF}})) = \sum_{i\in\mathcal{C}, k\in\mathcal{K}} \mathcal{I}_{\text{RF}}(\hat{\tilde{\boldsymbol{\beta}}}^{\text{RF}})_{(i,k),(i,k)} ,\label{eq_trace_1}\\ 
    \text{Tr}(\mathcal{I}_{\text{MLE}}(\hat{\tilde{\boldsymbol{\beta}}}^{\text{MLE}})) = \sum_{i\in\mathcal{C}, k\in\mathcal{K}} \mathcal{I}_{\text{MLE}}(\hat{\tilde{\boldsymbol{\beta}}}^{\text{MLE}})_{(i,k),(i,k)} .\label{eq_trace_2}
\end{align}
Comparing Eqs. \ref{eq_trace_1} and \ref{eq_trace_2}, to argue $ \text{Tr}(\mathcal{I}_{\text{RF}}(\hat{\tilde{\boldsymbol{\beta}}}^{\text{RF}})) \geq  \text{Tr}(\mathcal{I}_{\text{MLE}}(\hat{\tilde{\boldsymbol{\beta}}}^{\text{MLE}}))$, we need to satisfy two requirements:
\begin{itemize}
    \item \textbf{Requirement 1}: For the shared term $\left(\boldsymbol{x}_n^{(k)}\right)^2$, the coefficients are $\sum_{i\in\mathcal{C}}\tilde{P}_{i,n}^{\text{Norm}}\cdot(1-\tilde{P}_{i,n}^{\text{Norm}})$ and $\sum_{i\in\mathcal{C}}\tilde{P}_{i,n}\cdot(1-\tilde{P}_{i,n})$ for Eqs. \ref{eq_trace_1} and \ref{eq_trace_2}, respectively. We need to show $\sum_{i\in\mathcal{C}}\tilde{P}_{i,n}^{\text{Norm}}(\hat{\tilde{\boldsymbol{\beta}}}^{\text{RF}})\cdot(1-\tilde{P}_{i,n}^{\text{Norm}}(\hat{\tilde{\boldsymbol{\beta}}}^{\text{RF}})) \geq \sum_{i\in\mathcal{C}}\tilde{P}_{i,n}(\hat{\tilde{\boldsymbol{\beta}}}^{\text{MLE}})\cdot(1-\tilde{P}_{i,n}(\hat{\tilde{\boldsymbol{\beta}}}^{\text{MLE}}))$
    \item \textbf{Requirement 2}: The remaining parts in Eq. \ref{eq_trace_1} (excluding the shared term $\left(\boldsymbol{x}_n^{(k)}\right)^2$) are $\sum_{i\in\mathcal{C}, k\in\mathcal{K}} \sum_{n\in\mathcal{N}} 2\cdot \rho_{n}\cdot \left(\boldsymbol{x}_n^{(k)}\right)\cdot\left(\nabla_{\tilde{\bb}_i}\norm{\tilde{\bb}_i}_q\right)_k
    + (\rho_{n})^2\cdot \left(\nabla_{\tilde{\bb}_i}\norm{\tilde{\bb}_i}_q\right)_k^2$
    and $\sum_{i\in\mathcal{C}, k\in\mathcal{K}} \sum_{n\in\mathcal{N}}\tilde{P}_{i,n}^{\text{Norm}}\cdot\rho_{n}\cdot\left(\nabla^2_{\tilde{\bb}_i}\norm{\tilde{\bb}_i}_q \right)_{k,k}$. We need to show they are both non-negative. As $\ell_q$ norm is a convex function, its Hessian matrix is positive semi-definite. Therefore, $\left(\nabla^2_{\tilde{\bb}_i}\norm{\tilde{\bb}_i}_q \right)_{k,k} \geq 0$. So the second remaining part is non-negative. We only need to show the first part is non-negative.
\end{itemize}

Recall our preliminary for the proposition is $\rho_n$ is large enough. Its specific conditions are as follows:
\begin{itemize}
    \item \textbf{Condition 1}: $\rho_n$ is large enough such that $\sum_{i\in\mathcal{C}}\tilde{P}_{i,n}^{\text{Norm}}(\hat{\tilde{\boldsymbol{\beta}}}^{\text{RF}})\cdot(1-\tilde{P}_{i,n}^{\text{Norm}}(\hat{\tilde{\boldsymbol{\beta}}}^{\text{RF}})) \geq \sum_{i\in\mathcal{C}}\tilde{P}_{i,n}(\hat{\tilde{\boldsymbol{\beta}}}^{\text{MLE}})\cdot(1-\tilde{P}_{i,n}(\hat{\tilde{\boldsymbol{\beta}}}^{\text{MLE}}))$. We will explain later why this condition is related to the scale of $\rho_n$.
    \item \textbf{Condition 2}: $\rho_n$ is large enough such that 
    \begin{align}
        \rho_{n}  \geq 2\left|\frac{\boldsymbol{x}_n^{(k)}}{\left(\nabla_{\tilde{\bb}_i=\hat{\tilde{\boldsymbol{\beta}}}^{\text{RF}}_i}\norm{\tilde{\bb}_i}_q\right)_k} \right| .
    \end{align}
\end{itemize}
The explanation for Condition 1 is as follows. For function with form $\sum_i p_i(1-p_i)$ and $\sum_i p_i = 1$. The maximum value will be achieved when every $p_i$ is the same.
When $\rho_n$ is large enough, based on Proposition \ref{prop_inconsist}, $\hat{\tilde{\bb}}^{\text{RF}}$ will shrink toward 0, which will make $\tilde{P}_{i,n}^{\text{Norm}}$ closer to $\frac{1}{|\mathcal{C}|}$ (i.e., more similar to each other) than $\tilde{P}_{i,n}$. Therefore, Condition 1 can be achieved when $\rho_n$ is large enough. Based on condition 2, we can easily derive that:
\begin{align}\label{eq_rf_remain_part}
    2\rho_{n}\cdot \left(\nabla_{\tilde{\bb}_i=\hat{\tilde{\boldsymbol{\beta}}}^{\text{RF}}_i}\norm{\tilde{\bb}_i}_q\right)_k 
 + (\rho_{n})^2\cdot \left(\nabla_{\tilde{\bb}_i=\hat{\tilde{\boldsymbol{\beta}}}^{\text{RF}}_i}\norm{\tilde{\bb}_i}_q\right)_k^2 \geq 0 .
\end{align}
Eq. \ref{eq_rf_remain_part} implies the first remaining part listed in Requirement 2 is non-negative. Since Condition 1 directly implies Requirement 1, we have satisfied all requirements needed for claiming $ \text{Tr}(\mathcal{I}_{\text{RF}}(\hat{\tilde{\boldsymbol{\beta}}}^{\text{RF}})) \geq  \text{Tr}(\mathcal{I}_{\text{MLE}}(\hat{\tilde{\boldsymbol{\beta}}}^{\text{MLE}}))$

% 2) For off-diagonal terms:
% \begin{align}
%    \frac{\partial^2 \mathcal{LL}^{\text{RF}}_n}{\partial \tilde{\bb}_i\partial \tilde{\bb}_j} =&  \tilde{P}_{i,n}^{\text{Norm}}\cdot\tilde{P}_{j,n}^{\text{Norm}}\cdot\left(\boldsymbol{x}_n + \rho_{n}\cdot \nabla_{\tilde{\bb}_j}\norm{\tilde{\bb}_j}_q \right)\cdot \left(\boldsymbol{x}_n + \rho_{n}\cdot \nabla_{\tilde{\bb}_i}\norm{\tilde{\bb}_i}_q \right)^T 
% \end{align}

% Similarly, for the original MNL model:
% \begin{align}
%    &\frac{\partial^2 \mathcal{LL}^{\text{MLE}}_n}{\partial (\tilde{\bb}_i)^2} =  -\tilde{P}_{i,n}\cdot(1-\tilde{P}_{i,n})\cdot\boldsymbol{x}_n\cdot \boldsymbol{x}_n^T \\
%    &\frac{\partial^2 \mathcal{LL}^{\text{MLE}}_n}{\partial \tilde{\bb}_i\partial \tilde{\bb}_j} =  \tilde{P}_{i,n}\cdot\tilde{P}_{j,n}\cdot\boldsymbol{x}_n\cdot \boldsymbol{x}_n^T
% \end{align}
\end{proof}
Proposition \ref{prop_info_matrix} implies that, though the robust-feature MNL estimator is biased, its Fisher information matrix is larger (in terms of traces) than the original MLE estimator. Using the generalized Cramer-Rao bound \citep{cramer2016mathematical, rao1992information}, we have:
\begin{align}
    \text{Var}[\hat{\bb}^{\text{RF}}] \geq \text{Diag}\left(\mathcal{I}_{\text{RF}}(\boldsymbol{\hat{\boldsymbol{\beta}}^{\text{RF}}})^{-1}\cdot\left(\boldsymbol{I} + \boldsymbol{b}\boldsymbol{b}^T\right)\right),\quad \text{Var}[\hat{\bb}^{\text{MLE}}] \geq \text{Diag}\left(\mathcal{I}_{\text{RF}}(\boldsymbol{\hat{\boldsymbol{\beta}}^{\text{MLE}}})^{-1}\right) ,
\end{align} 
where $\boldsymbol{b}$ is the gradient of the bias term (usually non-tractable). Therefore, a larger trace of the Fisher information matrix usually implies lower estimated variance \citep{pukelsheim2006optimal}. However, the rigorous proof is difficult given the complex form of the Hessian matrix and the non-tractable bias term. 

\subsection{Statistical properties of the robust label estimators}
The analysis of robust label estimators is similar to robust feature estimators. We will show that the estimated parameter $\hat{\boldsymbol{\beta}}^{\text{RL}}$ depends on external hyper-parameter $\Gamma$, thus is inconsistent, and also show that it has a larger Fisher information matrix.

\begin{prop}\label{prop_inconsist_RL}
Let $\hat{\boldsymbol{\beta}}^{\text{RL}}$ be as the solution of Eq. \ref{eq_robust_label_mnl} (robust label MNL problem). $\hat{\boldsymbol{\beta}}^{\text{RL}}$ is an inconsistent estimate of the actual parameter. The value of $\hat{\boldsymbol{\beta}}^{\text{RL}}$ depends on external hyper-parameter $\Gamma$. A larger value of $\Gamma$ will shrink the scale of $\hat{\boldsymbol{\beta}}^{\text{RL}}$. 
\end{prop}
\begin{proof}
Let us decompose the original robust-label MNL problem (Eq. \ref{eq_original_rl}) as:
\begin{align}\label{eq_decomp_rl}
    \max_{\bb\in\mathcal{B}} \left(\sum_{n\in\mathcal{N}}\sum_{i\in\mathcal{C}_n} y_{n,i} \cdot \log(P_{n,i}(\boldsymbol{\beta})) + \min_{\boldsymbol{\Delta y}\in\mathcal{U}(\Gamma) }\left(\sum_{n\in\mathcal{N}}  \sum_{i\in\mathcal{C}_n}\Delta y_{n,i} \cdot \log(P_{n,i}(\boldsymbol{\beta}))\right) \right) .
\end{align}
From the definition of $\mathcal{U}(\Gamma)$, we need either have $(\Delta y_{n,j} = 0, \forall j\in\mathcal{C}_n)$ or $(\Delta y_{n,I_n} = -1$, one of the $\Delta y_{n,j} = 1, j\neq I_n)$. To minimize $\sum_{i\in\mathcal{C}_n}\Delta y_{n,i} \cdot \log(P_{n,i}(\boldsymbol{\beta}))$, we need to consider the relationship of $P_{n,j}$ for $j\in\mathcal{C}_n$. For this single sample $n$, one could easily derive the optimal selection of $\Delta y_{n,j}$ is:
\begin{subequations}
\begin{align}
   & \Delta y_{n,J_n^*} = 1,\; \Delta y_{n,I_n} = -1\quad \text{if $P_{n,I_n}(\boldsymbol{\beta}) > P_{n,j^*}(\boldsymbol{\beta})$, where $J_n^*=\argmin_{j}\{P_{n,j}: j\in\mathcal{C}_n\setminus\{I_n\}  \}$} \label{eq_opt_cond1}\\
   & \Delta y_{n,j} = 0, \forall j\in\mathcal{C}_n \quad \text{Otherwise} \label{eq_opt_cond2}
\end{align}
\end{subequations}
For $|\mathcal{N}|$ samples, the optimal solution will be selecting top $\Gamma$ samples with the lowest $ \left(\log (P_{n,J_n^*})- \log (P_{n,I_n})\right)$ to assign $\Delta y_{n,I_n} = -1, \Delta y_{n,J_n^*}=1$, and others to 0. Define the set of these selected top $\Gamma$ samples as $\mathcal{N}^{\text{Top}}(\bb;\Gamma)$. The optimal objective function of the inner minimization problem will be 
\begin{align}
   R(\bb;\Gamma) = \min_{\boldsymbol{\Delta y}\in\mathcal{U}(\Gamma)} \sum_{n\in\mathcal{N}}\sum_{i\in\mathcal{C}_n}\Delta y_{n,i} \cdot \log(P_{n,i}(\boldsymbol{\beta})) 
    = \sum_{n\in\mathcal{N}^{\text{Top}}(\bb;\Gamma)} \log (P_{n,J_n^*})- \log (P_{n,I_n}) .
\end{align}
Then Eq. \ref{eq_decomp_rl} can be rewritten as:
\begin{align}
    \max_{\bb\in\mathcal{B}} \left(\sum_{n\in\mathcal{N}}\sum_{i\in\mathcal{C}_n} y_{n,i} \cdot \log(P_{n,i}(\boldsymbol{\beta}))\right) + R(\bb;\Gamma) .\label{eq_robust_label_MNL_gammainf}
\end{align}
Therefore, $R(\bb;\Gamma)$ can be treated as a regularization term. It will penalize extreme value of $\bb$ because that could make $P_{n,J_n^*} \rightarrow 0$ and $R(\bb;\Gamma) \rightarrow -\infty$. The maximal value of $R(\bb;\Gamma)$ is achieved at $\bb = 0$ because, in this case, $P_{n,i} = P_{n,j}, \forall i,j \in\mathcal{C}_n$ and $R(\bb;\Gamma)=0$. Therefore, this regularization term tends to shrink $\bb$ toward 0. Thus, the robust-label estimator is biased. More importantly, the bias will not diminish with an increase in samples. Reversely, it will increase because more samples indicate the selection of top $\Gamma$ samples could have a lower value of $\log (P_{n,J_n^*})- \log (P_{n,I_n})$. Therefore, the robust-label MNL estimator is inconsistent. It is worth noting that when $\Gamma \rightarrow \infty$, $R(\bb;\Gamma)$ will not go to infinity (thus $\bb$ will not be 0), because the physical meaning of $\Gamma$ (i.e., the maximum number of label errors) implies it is only effective when $\Gamma \leq |\mathcal{N}|$. 
\end{proof}

\begin{prop}\label{prop_info_matrix_rl}
Let $\mathcal{I}_{\text{RL}}(\cdot)$ and $\mathcal{I}_{\text{MLE}}(\cdot)$ be the Fisher information matrix of the robust-label MNL estimator and the original MLE estimator, respectively. We have
\begin{align}
    \text{Tr}\left(\mathcal{I}_{\text{RL}}(\boldsymbol{\hat{\boldsymbol{\beta}}^{\text{RL}}})\right) \geq     \text{Tr}\left(\mathcal{I}_{\text{MLE}}(\boldsymbol{\hat{\boldsymbol{\beta}}^{\text{MLE}}})\right) .
\end{align}
where $\boldsymbol{\hat{\boldsymbol{\beta}}^{\text{RL}}}$ and $\boldsymbol{\hat{\boldsymbol{\beta}}^{\text{MLE}}}$ are corresponding estimated parameters. $\text{Tr}(\cdot)$ is the trace of a matrix.
\end{prop}
\begin{proof}
Looking at the rewritten form of the robust-label MNL problem in Eq. \ref{eq_robust_label_MNL_gammainf}. We have:
\begin{align}
    \mathcal{I}_{\text{RL}}(\boldsymbol{\hat{\boldsymbol{\beta}}^{\text{RL}}}) = \mathcal{I}_{\text{MLE}}(\boldsymbol{\hat{\boldsymbol{\beta}}^{\text{MLE}}}) - \nabla^2_{\bb={\hat{\boldsymbol{\beta}}^{\text{RL}}}} R(\bb;\Gamma) .
\end{align}
Notice that $\sum_{n\in\mathcal{N}}\sum_{i\in\mathcal{C}_n}\Delta y_{n,i} \cdot \log(P_{n,i}(\boldsymbol{\beta}))$ is jointly concave for both $\boldsymbol{\Delta y}$ and $\bb$. The minimization operator over $\text{Conv}\left(\mathcal{U}(\Gamma)\right)$ preserve concavity. Therefore, $R(\bb;\Gamma)$ is concave. And $- \nabla^2_{\bb} R(\bb;\Gamma) $ is positive semi-definite with $\text{Tr}(- \nabla^2_{\bb={\hat{\boldsymbol{\beta}}^{\text{RL}}}} R(\bb;\Gamma)) \geq 0$. Therefore,  $\text{Tr}\left(\mathcal{I}_{\text{RL}}(\boldsymbol{\hat{\boldsymbol{\beta}}^{\text{RL}}})\right) \geq     \text{Tr}\left(\mathcal{I}_{\text{MLE}}(\boldsymbol{\hat{\boldsymbol{\beta}}^{\text{MLE}}})\right)$

\end{proof}

It is worth noting that, compared to the proof of the robust-feature MNL model (Proposition \ref{prop_info_matrix}), there is no requirement for regularization parameters. 

\subsection{Out-of-sample prediction performance analysis}
\label{sec_out_sample_pred}

We have shown that both robust-feature and label estimators tend to shrink the scale of the estimated $\bb$, and may have lower variance. In this section, we show that when testing data have uncertainties, the introduced errors are larger for a larger scale of $\bb$. 

Consider a probabilistic classification problem for the bias-variance decomposition. Define $\sigma(\boldsymbol{x}|\bb)_i := \frac{\exp\left({\bb}_i^T \boldsymbol{x} \right)}{\sum_{j\in\mathcal{C}}\exp\left({\bb}_j^T \boldsymbol{x} \right)}$. Given a testing sample $\boldsymbol{x}$, its true probability of choosing alternative $i$ is $P_i$, where $P_i(\boldsymbol{x}) = \sigma(\boldsymbol{x}|{\bb}^{\text{True}})_i + \epsilon_i$. $\epsilon_i$ are random errors with zero mean. For a given estimated parameter $\hat{\bb}$, define the estimated probabilities as $\hat{P}_i(\boldsymbol{x}|\hat{\bb}) := \sigma(\boldsymbol{x}|\hat{\bb})_i$. Then one could quantify the estimation error of the probabilistic classification problem as $\mathbb{E}\left[\sum_{i\in\mathcal{C}}|P_i(\boldsymbol{x}) - \hat{P}_i(\boldsymbol{x}\mid\hat{\bb})|\right] = \mathbb{E}\left[\norm{\boldsymbol{P}(\boldsymbol{x}) -\hat{\boldsymbol{P}}(\boldsymbol{x}\mid\hat{\bb})}_1 \right]$, where $\boldsymbol{P}:=(P_i)_{i\in\mathcal{C}}$ and $\hat{\boldsymbol{P}}:=(\hat{P}_i)_{i\in\mathcal{C}}$.

\begin{prop}\label{prop_error_bound}
    Given a testing sample $\boldsymbol{x}$, assume we can only observe the contaminated data $\tilde{\boldsymbol{x}} = \boldsymbol{x} + \Delta \boldsymbol{x}$, where $\Delta \boldsymbol{x}$ is a random variable. Then, the prediction error for a given estimated $\hat{\bb}$ satisfies:
    \begin{align}
        \mathbb{E}\left[\norm{\boldsymbol{P}(\boldsymbol{x}) -\hat{\boldsymbol{P}}(\tilde{\boldsymbol{x}}\mid\hat{\bb})}_1 \right] \leq \underbrace{ \mathbb{E}\left[\norm{\boldsymbol{P}(\boldsymbol{x}) -\hat{\boldsymbol{P}}({\boldsymbol{x}}\mid\hat{\bb})}_1 \right]}_{\text{Errors without data perturbations}} +L\cdot \mathbb{E}\left[\norm{\boldsymbol{\Delta x}}_2\right] \cdot \hat{\norm{\bb}}_2 ,
    \end{align}
where $L$ is the Lipschitz constant of the softmax function under $\ell_1$ norm.
\end{prop}

\begin{proof}

    \begin{align}
        \mathbb{E}\left[\norm{\boldsymbol{P}(\boldsymbol{x}) -\hat{\boldsymbol{P}}(\tilde{\boldsymbol{x}}\mid\hat{\bb})}_1 \right] 
        &\leq \mathbb{E}\left[\norm{\boldsymbol{P}(\boldsymbol{x}) -\hat{\boldsymbol{P}}({\boldsymbol{x}}\mid\hat{\bb})}_1 \right] + \mathbb{E}\left[\norm{\hat{\boldsymbol{P}}({\boldsymbol{x}}\mid\hat{\bb}) -\hat{\boldsymbol{P}}(\tilde{\boldsymbol{x}}\mid\hat{\bb})}_1 \right] \notag\\
        & =  \mathbb{E}\left[\norm{\boldsymbol{P}(\boldsymbol{x}) -\hat{\boldsymbol{P}}({\boldsymbol{x}}\mid\hat{\bb})}_1 \right] + \mathbb{E}\left[\norm{\sigma(\boldsymbol{x}|\hat{\bb}) -\sigma(\boldsymbol{x}+\boldsymbol{\Delta x}|\hat{\bb})}_1 \right] \notag \\ 
        & \leq  \mathbb{E}\left[\norm{\boldsymbol{P}(\boldsymbol{x}) -\hat{\boldsymbol{P}}({\boldsymbol{x}}\mid\hat{\bb})}_1 \right] +  L\cdot \mathbb{E}\left[\norm{(\hat{\bb}_i^T \boldsymbol{\Delta x})_{i\in\mathcal{C}}}_1\right] \notag\\
        & \leq \mathbb{E}\left[\norm{\boldsymbol{P}(\boldsymbol{x}) -\hat{\boldsymbol{P}}({\boldsymbol{x}}\mid\hat{\bb})}_1 \right] +  L\cdot \mathbb{E}\left[\norm{\boldsymbol{\Delta x}}_2\right] \cdot \hat{\norm{\bb}}_2 ,
    \end{align}
where the first inequality is from the triangle inequality. The second inequality is from the Lipschitz continuity of the softmax function under the $\ell_1$ norm. $L$ is the corresponding constant. The last inequality is from the Cauchy-Schwarz inequality.

\end{proof}

From Proposition \ref{prop_error_bound}, we see that the upper bound of prediction errors when data has uncertainties has two parts: one is the error without data perturbations and another one is proportional to the scale of estimated $\bb$. Since robust estimators tend to shrink the scale of $\bb$, they are preferred for cases with data perturbations. 

Note that the error without data perturbations follows the typical bias-variances trade-off \citep{kohavi1996bias, vapnik1999nature}. The robust estimators have higher bias and usually smaller variances. It could potentially improve the model's generalizability with proper selection of $\rho_n$ like regularization in machine learning. Details on the bias-variances trade-off analysis can be found in \ref{append_bia_var}

\section{Numerical Experiments}\label{sec_case_study}

\subsection{Experiments design}

\subsubsection{Binary case study}\label{sec_binary_exp_setup}
The robust BNL is evaluated on the Singapore first- and last-mile travel mode choice data set \citep{mo2018impact}. The data set is part of Singapore's Household Interview Travel Survey (HITS) in 2012. The survey collects information on travel characteristics as well as individual sociodemographics. The first and last-mile trips are extracted from the whole trip train in the HITS. Besides travel characteristics, the built environment information is also included as they highly impact the mode choices. Data collection details can be found in \citet{mo2018impact}. The alternative travel modes for the first/last mile trips are walk and bus.

The whole data set contains a total of more than 24,000 observations. In the case study, we randomly select 1,000 samples as the training data set and another 1,000 samples as the testing set. Denote the raw training and testing data set as $\mathcal{D}^{\text{Train}}$ and $\mathcal{D}^{\text{Test}}$, respectively. In order to simulate data uncertainties, we generate the synthetic testing data set with errors as the following:
\begin{itemize}
    \item Step 1: Train a conventional BNL model in $\mathcal{D}^{\text{Test}}$ and assume that the obtained parameters $\boldsymbol{\beta}^{\text{Test}}$ are the  ``true'' behavior mechanism that individuals will follow.
    \item Step 2: For each individual $n$ in the testing data set, generate the synthetic choice $\hat{\boldsymbol{y}}_n$ based on $\boldsymbol{x}_n$ and $\boldsymbol{\beta}^{\text{Test}}$ (i.e., calculate the choice probabilities and randomly select one alternative based on the probabilities). 
    % The reason for generating synthetic choice based on $\boldsymbol{\beta}^{\text{Test}}$ is to allow us to compare the estimated parameters with the ``oracle'' one. 
    \item Step 3: Add artificial errors to the generated data to get the final synthetic testing set with errors: $\Tilde{\boldsymbol{y}}_n = \hat{\boldsymbol{y}}_n + \boldsymbol{\Delta{y}}$ and $\Tilde{\boldsymbol{x}}_n = {\boldsymbol{x}}_n + \boldsymbol{\Delta{x}}$. Let the synthetic testing data set with errors be $\tilde{\mathcal{D}}^{\text{Test}}$.
\end{itemize}
Specifically, the random errors $\boldsymbol{\Delta{x}}$ and $\boldsymbol{\Delta{y}}$ are generated as follows. $\boldsymbol{\Delta{x}}$ are drawn from a uniform distribution $\mathtt{U}[-0.3\Bar{\boldsymbol{x}}, 0.3\Bar{\boldsymbol{x}}]$, where $\Bar{\boldsymbol{x}} = {\sum_{n\in\mathcal{N}} \boldsymbol{x}_n} / {|\mathcal{N}|}$ is the average value of features (i.e., we perturb the features by maximally 30\%). $\boldsymbol{\Delta{y}}$ is a perturbation to the labels such that, with 10\% probability, the label $\boldsymbol{y}_n$ is replaced by a randomly-selected alternative in $\mathcal{C}_n$.

``Walk'' is set as the base mode. All models will be trained or estimated in the training data set $\mathcal{D}^{\text{Train}}$, and tested in the synthetic testing data $\tilde{\mathcal{D}}^{\text{Test}}$ to evaluate their performances. The data generation process is replicated 30 times and all models are trained and evaluated in those 30 replications to reduce the impact of randomness.

\subsubsection{Multinomial case study}
The robust MNL is evaluated on the Swissmetro stated preference survey data set \citep{bierlaire2001acceptance}. The survey aims to analyze the impact of travel modal innovation in transportation, represented by the Swissmetro, a revolutionary maglev underground system, against the usual transport modes represented by car and train. The data contains 1,004 individuals with 9,036 responses. Users are asked to select from three travel modes (train, car, and Swissmetro) given the corresponding travel attributes. ``Train'' is set as the base mode. The training and testing data set are generated in the same way as the binary case study with 1,000 randomly selected samples for both training and testing data sets. 

\subsubsection{Definition of metrics}
Two major metrics are reported in the case study: log-likelihood (LL) and accuracy. Given a data set $\mathcal{D}=\{(\boldsymbol{x}_n, \boldsymbol{y}_n): \forall n \in\mathcal{N}_{\mathcal{D}}\}$, $\mathcal{N}_{\mathcal{D}}$ is the set of sample index of the data set $\mathcal{D}$. We have:
\begin{align}
    \text{Accuracy}(\mathcal{D}) & = \frac{\mathbbm{1}\big(\argmax_{i\in\mathcal{C}_n} \{\sigma({\boldsymbol{x}_n|\hat{\bb}})_i\}=\argmax_{i\in\mathcal{C}_n} \{{y}_{n,i}\} \big)}{|\mathcal{N}_{\mathcal{D}}|} \\
    \text{Log-likelihood}(\mathcal{D}) & = \sum_{n\in\mathcal{N}} \sum_{i\in\mathcal{C}_n} y_{n,i}\cdot \log\left(\sigma({\boldsymbol{x}_n|\hat{\bb}})_i\right)
\end{align}
where $\mathbbm{1}(\cdot)$ is the indicator function (equal to 1 if true, otherwise 0). We will evaluate the accuracy and log-likelihood for both $\mathcal{D}^{\text{Train}}$ and $\tilde{\mathcal{D}}^{\text{Test}}$. Note that the robust formulations are only used to estimate parameters. Once the parameters are obtained. Log-likelihood is evaluated on its original definition. 

\subsection{Experiment results}
\subsubsection{Binary case study}
In the case study, we set $p = 2$, thus $q$ also equals 2. The final robustness term for robust BNL model (Eq. \ref{eq_robust_bnl}) is $\rho_{n} \norm{\bb_{I_n} - \bb_{J_n}}_2$. We also assume all individuals have the same uncertainty budget $\rho_n$ (i.e., $\rho_n = \rho_m$ for all $n,m\in\mathcal{N}$). Table \ref{tab_res_bnl} shows the training and testing accuracy and log-likelihood (LL) with respect to different values of $\rho_n$ and $\Gamma$. 

\begin{table}[H]
\centering
\caption{Results for robust BNL models}\label{tab_res_bnl}
\resizebox{\linewidth}{!}
{
\begin{tabular}{@{}llcccc@{}}
\toprule
\textbf{Model}  & \textbf{Parameters}   & \textbf{Training accuracy} & \textbf{Training LL} & \textbf{Testing accuracy} & \textbf{Testing LL} \\ \midrule
Binary Logit & -  & 0.957 ($\pm$0.010) & -152.6 ($\pm$22.6) & 0.879 ($\pm$0.014) & -358.0 ($\pm$77.1) \\ \midrule
\multirow{5}{*}{Robust-feature}& $\rho_n = 0.001$& 0.957 ($\pm$0.010) & -152.6 ($\pm$22.6) & 0.880 ($\pm$0.013) & -353.8 ($\pm$75.5) \\
& $\rho_n = 0.01$&  0.957 ($\pm$0.010) & -153.1 ($\pm$22.5) & 0.884 ($\pm$0.014) & -337.5 ($\pm$67.6) \\
& $\rho_n = 0.1$&  0.951 ($\pm$0.009) & -163.3 ($\pm$22.0) & \cellcolor{gray!25}0.890 ($\pm$0.012) & -289.3 ($\pm$33.0) \\
& $\rho_n = 0.2$&       0.940 ($\pm$0.008) & -188.8 ($\pm$21.8) & 0.888 ($\pm$0.011) & \cellcolor{gray!25}-273.2 ($\pm$21.9) \\ 
& $\rho_n = 0.3$&     0.932 ($\pm$0.008) & -230.7 ($\pm$21.7) & 0.883 ($\pm$0.012) & -287.3 ($\pm$19.0)
\\ \midrule
\multirow{5}{*}{Robust-label}   
                                & $\Gamma = 1$ & 0.954 ($\pm$0.009) &   -157.4 ($\pm$21.6) &   0.881 ($\pm$0.012) & -313.7 ($\pm$43.8)\\
                                & $\Gamma = 1.5$ &   0.954 ($\pm$0.009)  & -160.3 ($\pm$21.2) & \cellcolor{gray!25}0.882 ($\pm$0.012)  & -307.4 ($\pm$38.4) \\
                                & $\Gamma = 2$ &   0.953 ($\pm$0.009)   &-163.6 ($\pm$20.9)   & 0.881 ($\pm$0.011) & -303.8 ($\pm$34.1)  \\           
                                & $\Gamma = 2.5$ &   0.953 ($\pm$0.009)   & -167.0 ($\pm$20.6) &0.879 ($\pm$0.011)  &  -302.3 ($\pm$31.9)\\
                                & $\Gamma = 3$ &     0.952 ($\pm$0.009)  & -170.5 ($\pm$20.4)& 0.878 ($\pm$0.012)  & \cellcolor{gray!25} -301.7 ($\pm$30.5) \\ \bottomrule
\multicolumn{6}{l}{\begin{tabular}[c]{l} 
- All results are the average values of 30 replications. Values in the parentheses are standard deviations \\
- Best models are highlighted with gray cells
\end{tabular}}
\end{tabular}
}
\end{table}

We find that, with larger values of $\rho_n$ and $\Gamma$, the training accuracy keeps decreasing and training LL becomes smaller, showing worse goodness of fit in the training data set. This is as expected, because the objective functions are weighted more on the robustness term. However, in the testing set, the both robust feature and label models perform better than the typical binary logit model. The best robust feature BNL ($\rho_n = 0.1$) has a testing accuracy of 0.890 and the best robust label BNL ($\Gamma = 1.5$) has a testing accuracy of 0.882, while the binary logit model's testing accuracy is 0.879. The improvement of testing LL is even higher for robust models.

As shown in Section \ref{sec_stats}, the prediction errors of the robust model can be decomposed into 3 components: 1) bias, 2) variances, and 3) errors from data. The robust models have higher bias, (likely) lower variance, and lower errors from data. Therefore, its performance will be inherently determined by the value of $\rho_n$ and $\Gamma$. If $\rho_n$ and $\Gamma$ are too large, we will have too conservative uncertainty sets, and the bias term will be too large, resulting in higher prediction errors of the robust models. With proper values of $\rho_n$ and $\Gamma$, we could leverage the lower variances and lower errors from data to offset the higher bias and get better prediction performance.

The estimated $\bb$ values of selected features are compared in Table \ref{tab_beta_label}. Walk time and bus in-vehicle time (IVT) are alternative-specific variables. Walking is set as the base mode. Hence, the alternative specific constant (ASC bus), distance to subway station (Dis. to sub.), and bus station accessibility (Bus access.) are only put in the utility function of the bus. We find that the estimated $\bb$ of the robust-label DCM works like ``regularization'', which shrinks the estimated $\bb$ value towards 0 (compared to the binary logit model) as we expected in Proposition \ref{prop_inconsist}. This provides another explanation of the good performance for out-of-sample prediction for the robust BNL model: To achieve robustness under data uncertainties, the model tends to estimate ``smaller'' (absolute) values of $\bb$. The ``smaller'' $\bb$ has better generalizability towards predicting samples that have different patterns with the training data set (as discussed in Proposition \ref{prop_error_bound}). This mechanism is similar to how regularization works in machine learning studies. 

\begin{table}[H]
\centering
\caption{Selected estimated $\bb$ for robust BNL models}\label{tab_beta_label}
\resizebox{\linewidth}{!}
{
\begin{tabular}{@{}llccccc@{}}
\toprule
\textbf{Model}  & \textbf{Parameters}   & \textbf{Walk time} & \textbf{Bus IVT} & \textbf{Dist. to sub.} &  \textbf{Bus access.}  & \textbf{ASC bus}\\ \midrule
Binary Logit & -  & -3.21 ($\pm$0.56) & -5.29 ($\pm$1.11) & 6.17 ($\pm$1.61) & 0.95 ($\pm$0.91) & -6.97 ($\pm$2.36)  \\ \midrule
\multirow{5}{*}{Robust-feature}& $\rho_n = 0.001$ & -3.21 ($\pm$0.55) & -5.29 ($\pm$1.11) & 6.15 ($\pm$1.60) & 0.92 ($\pm$0.89)  & -6.72 ($\pm$2.25) \\
& $\rho_n = 0.01$&  -3.19 ($\pm$0.52) & -5.19 ($\pm$1.06) & 5.98 ($\pm$1.54) & 0.69 ($\pm$0.78)  & -5.35 ($\pm$1.58) \\
& $\rho_n = 0.1$& -3.10 ($\pm$0.30) & -3.79 ($\pm$0.65) & 3.90 ($\pm$0.84) & -0.48 ($\pm$0.25) & -2.54 ($\pm$0.57) \\
& $\rho_n = 0.2$&  -2.62 ($\pm$0.23) & -2.32 ($\pm$0.33) & 2.19 ($\pm$0.34) & -0.70 ($\pm$0.11) & -1.39 ($\pm$0.25) \\
& $\rho_n = 0.3$& -1.93 ($\pm$0.16) & -1.33 ($\pm$0.17) & 1.32 ($\pm$0.15) & -0.56 ($\pm$0.07) & -0.80 ($\pm$0.12) 
\\ \midrule
\multirow{5}{*}{Robust-label}   
                                & $\Gamma = 1$ &  -2.79 ($\pm$0.45) & -3.99 ($\pm$0.73) & 4.52 ($\pm$1.03) & 0.80 ($\pm$0.61) & -6.05 ($\pm$1.69) \\
                                & $\Gamma = 1.5$ & -2.63 ($\pm$0.41) & -3.68 ($\pm$0.66) & 4.11 ($\pm$0.93) & 0.78 ($\pm$0.58) & -5.88 ($\pm$1.64) \\ 
                                & $\Gamma = 2$ &       -2.50 ($\pm$0.38) & -3.44 ($\pm$0.60) & 3.77 ($\pm$0.86) & 0.77 ($\pm$0.58) & -5.72 ($\pm$1.61) \\  
                                & $\Gamma = 2.5$ & -2.39 ($\pm$0.36) & -3.23 ($\pm$0.56) & 3.49 ($\pm$0.80) & 0.77 ($\pm$0.57) & -5.69 ($\pm$1.51) \\ 
                                & $\Gamma = 3$ & -2.30 ($\pm$0.35) & -3.06 ($\pm$0.52) & 3.25 ($\pm$0.76) & 0.76 ($\pm$0.57) & -5.61 ($\pm$1.44)    \\ \bottomrule
\multicolumn{7}{l}{\begin{tabular}[c]{l} 
- All results are the average values of 30 replications. Values in the parentheses are standard deviations
\end{tabular}}
\end{tabular}
}
\end{table}

\subsubsection{Multinomial case study}
Similar to the binary case study, we set $p = 2$ (thus $q=2$) and make the robustness term for the robust MNL model (Eq. \ref{eq_robust_mnl_jesen}) be $\rho_{n} \norm{\bb_{j} - \bb_{I_n}}_2$. The results in terms of different values of $\rho_{n}$ and $\Gamma$ are shown in Table \ref{tab_res_mnl}. Note that the parameter sets are different from the binary case study as the data sets are different. 

\begin{table}[H]
\centering
\caption{Results for robust MNL models}\label{tab_res_mnl}
\resizebox{\linewidth}{!}
{
\begin{tabular}{@{}llcccc@{}}
\toprule
\textbf{Model}  & \textbf{Parameters}   & \textbf{Training accuracy} & \textbf{Training LL} & \textbf{Testing accuracy} & \textbf{Testing LL} \\ \midrule
MNL & -  & 0.591 ($\pm$0.014) & -871.5 ($\pm$15.3)  & 0.540 ($\pm$0.021) & -988.5 ($\pm$41.5)   \\ \midrule
\multirow{5}{*}{Robust-feature}
& $\rho_n = 0.001$& 0.590 ($\pm$0.014) & -871.6 ($\pm$15.3)  & 0.542 ($\pm$0.022) & -981.5 ($\pm$39.5)  \\
& $\rho_n = 0.01$&  0.588 ($\pm$0.014) & -874.7 ($\pm$15.3)  & 0.552 ($\pm$0.021) & -952.2 ($\pm$33.2)  \\
& $\rho_n = 0.1$& 0.585 ($\pm$0.013) & -923.7 ($\pm$15.2)  & \cellcolor{gray!25}0.565 ($\pm$0.019) & \cellcolor{gray!25}-942.0 ($\pm$16.8)  \\
& $\rho_n = 0.15$& 0.578 ($\pm$0.018) & -966.7 ($\pm$13.5)  & 0.558 ($\pm$0.018) & -975.8 ($\pm$15.0)  \\
&$\rho_n = 0.2$ & 0.503 ($\pm$0.037) & -1001.6 ($\pm$10.6) & 0.519 ($\pm$0.029) & -1007.8 ($\pm$14.0)
\\ \midrule
\multirow{5}{*}{Robust-label}   
                                & $\Gamma = 1$ & 0.589 ($\pm$0.013) & -873.9 ($\pm$15.1) & 0.544 ($\pm$0.022) & -972.6 ($\pm$36.7) \\
                                & $\Gamma = 10$ & 0.583 ($\pm$0.014) & -886.5 ($\pm$15.5) & 0.555 ($\pm$0.022) & \cellcolor{gray!25}-937.5 ($\pm$26.2) \\
                                & $\Gamma = 100$ &  0.579 ($\pm$0.013) & -941.5 ($\pm$13.1) & \cellcolor{gray!25}0.558 ($\pm$0.020) & -956.1 ($\pm$13.3) \\    
                                & $\Gamma = 200$ &  0.582 ($\pm$0.015) & -979.7 ($\pm$11.0) & 0.557 ($\pm$0.022) & -987.3 ($\pm$9.9)  \\
                                & $\Gamma = 300$ &  0.586 ($\pm$0.015) & -1007.2 ($\pm$9.5) & 0.556 ($\pm$0.021) & -1012.9 ($\pm$8.7) \\ \bottomrule
\multicolumn{6}{l}{\begin{tabular}[c]{l} 
- All results are the average values of 30 replications. Values in the parentheses are standard deviations \\
- Best models are highlighted with gray cells
\end{tabular}}
\end{tabular}
}
\end{table}

The results are similar to what we observe in the binary cases. In general, with reasonable settings of the robust budget parameter $\rho_n$ and $\Gamma$, we observe better testing accuracy and LL. The best robust feature model ($\rho_n = 0.1$) has a testing accuracy of 0.565 and the best robust label model ($\Gamma = 100$) has a testing accuracy of 0.558. While the conventional MNL model's testing accuracy is 0.540. It is worth noting that, compared to the binary case study, the best hyper-parameter for robust label models can be quite different. This implies that hyper-parameter tuning may be necessary for implementing the robust MNL model. 

The estimated $\bb$ values of alternative specific features (i.e., cost and travel time) are shown in Table \ref{tab_beta_mnl}. Similar to binary cases, we observe the shrinkage of parameter scales with the increasing robust budget parameters. This is consistent with our analysis in Remark \ref{remark_robust_label_gamma}. Note that for the robust-label optimization, when $\Gamma$ goes to $+\infty$, $\bb$ will not be 0 (see Remark \ref{remark_robust_label_gamma}). The final values will depend on the data set.

\begin{table}[H]
\centering
\caption{Selected estimated $\bb$ for robust MNL models}\label{tab_beta_mnl}
\resizebox{\linewidth}{!}
{
\begin{tabular}{@{}llcccccc@{}}
\toprule
\textbf{Model}  & \textbf{Parameters}   & \textbf{Train time} & \textbf{Train cost} & \textbf{Car time} &  \textbf{Car cost}  & \textbf{SM time} & \textbf{SM cost}\\ \midrule
MNL & -  & -1.34 ($\pm$0.15) & -0.13 ($\pm$0.04) & -1.27 ($\pm$0.33) & -0.10 ($\pm$0.03) & -1.08 ($\pm$0.25) & -0.50 ($\pm$0.28) \\ \midrule
\multirow{5}{*}{Robust-feature}& $\rho_n = 0.001$ & -1.34 ($\pm$0.15) & -0.12 ($\pm$0.03) & -1.26 ($\pm$0.33) & -0.09 ($\pm$0.03) & -1.07 ($\pm$0.25) & -0.49 ($\pm$0.28) \\
& $\rho_n = 0.01$& -1.26 ($\pm$0.14) & -0.06 ($\pm$0.02) & -1.19 ($\pm$0.30) & -0.04 ($\pm$0.01) & -1.00 ($\pm$0.23) & -0.42 ($\pm$0.25) \\
& $\rho_n = 0.1$& -0.69 ($\pm$0.08) & 0.04 ($\pm$0.01)  & -0.44 ($\pm$0.12) & 0.01 ($\pm$0.01)  & -0.33 ($\pm$0.09) & -0.09 ($\pm$0.07) \\
& $\rho_n = 0.15$& -0.45 ($\pm$0.06) & 0.06 ($\pm$0.01)  & -0.16 ($\pm$0.06) & 0.02 ($\pm$0.00)  & -0.11 ($\pm$0.04) & -0.01 ($\pm$0.02) \\
& $\rho_n = 0.2$& -0.29 ($\pm$0.04) & 0.07 ($\pm$0.01)  & -0.03 ($\pm$0.02) & 0.02 ($\pm$0.00)  & -0.02 ($\pm$0.01) & 0.00 ($\pm$0.00)  
\\ \midrule
\multirow{5}{*}{Robust-label}   
                                & $\Gamma = 1$ &  -1.20 ($\pm$0.17) & -0.13 ($\pm$0.04) & -1.07 ($\pm$0.34) & -0.10 ($\pm$0.03) & -0.74 ($\pm$0.21) & -0.74 ($\pm$0.26) \\
                                & $\Gamma = 10$ & -0.89 ($\pm$0.19) & -0.06 ($\pm$0.02) & -0.80 ($\pm$0.33) & -0.05 ($\pm$0.02) & -0.44 ($\pm$0.23) & -0.67 ($\pm$0.24) \\
                                & $\Gamma = 100$ &   -0.36 ($\pm$0.08) & -0.03 ($\pm$0.01) & -0.37 ($\pm$0.15) & -0.02 ($\pm$0.01) & -0.21 ($\pm$0.10) & -0.29 ($\pm$0.11) \\
                                & $\Gamma = 200$ &  -0.22 ($\pm$0.06) & -0.03 ($\pm$0.01) & -0.24 ($\pm$0.10) & -0.02 ($\pm$0.01) & -0.13 ($\pm$0.06) & -0.17 ($\pm$0.06) \\
                                & $\Gamma = 300$ & -0.16 ($\pm$0.06) & -0.03 ($\pm$0.01) & -0.18 ($\pm$0.08) & -0.01 ($\pm$0.01) & -0.10 ($\pm$0.04) & -0.10 ($\pm$0.04)    \\ \bottomrule
\multicolumn{7}{l}{\begin{tabular}[c]{l} 
- All results are the average values of 30 replications. Values in the parentheses are standard deviations
\end{tabular}}
\end{tabular}
}
\end{table}

\subsection{Impact of error generation}
The model is derived based on the assumption that all perturbations $\boldsymbol{\Delta x}_n$ are independent across individuals. It is worth exploring model performance if $\boldsymbol{\Delta x}_n$ are correlated given some system errors. In this section, we consider the following system errors to generate the testing data set. 
\begin{itemize}
 \item \textbf{Over-reporting of travel time}: For features of travel time, their corresponding $\boldsymbol{\Delta x}$ are drawn from a uniform distribution $\mathtt{U}[0, 0.3\Bar{\boldsymbol{x}}]$. Other perturbations follow the same settings as in Section \ref{sec_binary_exp_setup}. 
 \item \textbf{Under-reporting of travel time}. For features of travel time, their corresponding $\boldsymbol{\Delta x}$ are drawn from a uniform distribution $\mathtt{U}[-0.3\Bar{\boldsymbol{x}}, 0]$. Other perturbations follow the same settings as in Section \ref{sec_binary_exp_setup}.
 \item \textbf{Social desirability bias} (i.e., respondents might over-report environmentally friendly choices like walking or cycling). With 10\% probability, the actual selections of non-walk and non-bike modes are replaced by walk or bike modes. Other perturbations follow the same settings as in Section \ref{sec_binary_exp_setup}
\end{itemize}
Since only the Singapore HITS dataset is from a revealed preference survey with user-reported travel time, the HITS dataset is used for testing. The robust-feature models use $\ell_2$ norm with $\rho_n = 0.1$ and robust-label models use $\Gamma = 1.5$ as they are the best hyper-parameters in the previous experiment. The prediction performance is shown in Table \ref{tab_res_bnl_error}. We saw that though the robust models are derived from independent error assumptions, they also outperform the nominal model when there are systematic errors in the system (i.e., error terms are correlated).

\begin{table}[H]
\centering
\caption{Impact of different data error generations}\label{tab_res_bnl_error}
\resizebox{\linewidth}{!}
{
\begin{tabular}{@{}llcccc@{}}
\toprule
\textbf{Scenarios}  & \textbf{Models}   & \textbf{Training accuracy} & \textbf{Training LL} & \textbf{Testing accuracy} & \textbf{Testing LL} \\ \midrule
\multirow{3}{*}{Independent errors} & BNL  & 0.957 ($\pm$0.010) & -152.6 ($\pm$22.6) & 0.879 ($\pm$0.014) & -358.0 ($\pm$77.1) \\ 
 & Robust-feature  & 0.951 ($\pm$0.009) &-163.3 ($\pm$22.0) &0.890 ($\pm$0.012) & -289.3 ($\pm$33.0) \\ 
 & Robust-label  & 0.954 ($\pm$0.009) & -160.3 ($\pm$21.2) & 0.882 ($\pm$0.012) & -307.4 ($\pm$38.4) \\ 
\midrule
\multirow{3}{*}{Over-reporting}& BNL  & 0.957 ($\pm$0.010) & -152.6 ($\pm$22.6) & 0.906 ($\pm$0.013) & -279.8 ($\pm$58.6) \\ 
 & Robust-feature  & 0.951 ($\pm$0.009) &-163.3 ($\pm$22.0) &0.917 ($\pm$0.011) & -225.6 ($\pm$28.6) \\ 
 & Robust-label  & 0.954 ($\pm$0.009) & -160.3 ($\pm$21.2) & 0.909 ($\pm$0.013) & -247.4 ($\pm$37.8)
\\ \midrule
\multirow{3}{*}{Under-reporting}& BNL  & 0.957 ($\pm$0.010) & -152.6 ($\pm$22.6) & 0.881 ($\pm$0.018) & -326.2 ($\pm$70.3) \\ 
 & Robust-feature  & 0.951 ($\pm$0.009) &-163.3 ($\pm$22.0) &0.897 ($\pm$0.012) & -256.4 ($\pm$27.2) \\ 
 & Robust-label  & 0.954 ($\pm$0.009) & -160.3 ($\pm$21.2) & 0.886 ($\pm$0.016) & -284.9 ($\pm$42.2)  
\\ \midrule
\multirow{3}{*}{Social desirability}  &  BNL  & 0.957 ($\pm$0.010) & -152.6 ($\pm$22.6) & 0.860 ($\pm$0.015) & -514.0 ($\pm$111.2) \\ 
 & Robust-feature  & 0.951 ($\pm$0.009) &-163.3 ($\pm$22.0) &0.871 ($\pm$0.013) & -415.3 ($\pm$67.2) \\ 
 & Robust-label  & 0.954 ($\pm$0.009) & -160.3 ($\pm$21.2) & 0.863 ($\pm$0.013) & -425.5 ($\pm$63.8) \\    \bottomrule
\multicolumn{6}{l}{\begin{tabular}[c]{l} 
- All results are the average values of 30 replications. Values in the parentheses are standard deviations \\
- ``Independent errors'' represent the same setting as Section \ref{sec_binary_exp_setup}
\end{tabular}}
\end{tabular}
}
\end{table}

\subsection{Impact of norms in uncertainty sets}
The above case study is based on the $\ell_2$ norm for the robust-feature MNL model. In this section, we test how different norms will affect the prediction results. Only the MNL model with Swissmetro data is tested. We fix $\rho_n = 0.1$ for all testing cases as 0.1 is the best hyper-parameter for the $\ell_2$ norm. The implementation of $\ell_\infty$ is to define a new auxiliary decision variable $u_{n,i}$ to replace $\norm{\bb_j - \bb_{I_n}}_\infty$, and add $u_{n,i} \geq (\bb_i - \bb_{I_n})^{(k)}$ and $u_{n,i} \geq -(\bb_i - \bb_{I_n})^{(k)}$ $ \forall k\in\mathcal{K}$ to the constraints. The prediction performance is shown in Table \ref{tab_res_mnl_norm}. From the results, we see that the robust-feature models always outperform the nominal MNL model in out-of-sample prediction accuracy, no matter what norm we use. Since $\rho_n = 0.1$ is the best hyper-parameter for the $\ell_2$ norm. $\ell_2$ norm still gives the best testing accuracy. Note that for other norms, hyper-parameter tuning is needed to obtain better performance than current results.

\begin{table}[H]
\centering
\caption{Impact of different norms on model performance}\label{tab_res_mnl_norm}
\resizebox{\linewidth}{!}
{
\begin{tabular}{@{}llcccc@{}}
\toprule
\textbf{Model}  & \textbf{Norms}   & \textbf{Training accuracy} & \textbf{Training LL} & \textbf{Testing accuracy} & \textbf{Testing LL} \\ \midrule
MNL & -  & 0.591 ($\pm$0.014) & -871.5 ($\pm$15.3)  & 0.540 ($\pm$0.021) & -988.5 ($\pm$41.5)   \\ \midrule
\multirow{5}{*}{Robust-feature}
& $q = 2$& 0.585 ($\pm$0.013) & -923.7 ($\pm$15.2)  & \cellcolor{gray!25}0.565 ($\pm$0.019) & -942.0 ($\pm$16.8)  \\
& $q = 3$&  0.586 ($\pm$0.012) & -907.8 ($\pm$15.7)  & 0.563 ($\pm$0.020) & -935.8 ($\pm$19.0)  \\
& $q = 4$&0.586 ($\pm$0.012) &-902.4 ($\pm$15.8)  &0.560 ($\pm$0.019) & \cellcolor{gray!25}-935.5 ($\pm$20.3)  \\
& $q = 10$& 0.586 ($\pm$0.013) & -896.7 ($\pm$15.9)  & 0.557 ($\pm$0.016) & -936.9 ($\pm$21.9)  \\
&$q = 100$ & 0.588 ($\pm$0.014) & -894.2 ($\pm$14.7) & 0.556 ($\pm$0.017) & -936.9 ($\pm$22.9)\\

&$q = \infty$ & 0.587 ($\pm$0.014) & -894.7 ($\pm$16.0) &0.556 ($\pm$0.017) & -937.7 ($\pm$22.8)
\\
\bottomrule
\multicolumn{6}{l}{\begin{tabular}[c]{l} 
- All results are the average values of 30 replications. Values in the parentheses are standard deviations \\
- Best models are highlighted with gray cells
\end{tabular}}
\end{tabular}
}
\end{table}

% The estimated parameters for different norms are shown in Table \ref{tab_beta_mnl_norm}.

% \begin{table}[H]
% \centering
% \caption{Selected estimated $\bb$ for robust MNL models}\label{tab_beta_mnl_norm}
% \resizebox{\linewidth}{!}
% {
% \begin{tabular}{@{}llcccccc@{}}
% \toprule
% \textbf{Model}  & \textbf{Norms}   & \textbf{Train time} & \textbf{Train cost} & \textbf{Car time} &  \textbf{Car cost}  & \textbf{SM time} & \textbf{SM cost}\\ \midrule
% MNL & -  & -1.34 ($\pm$0.15) & -0.13 ($\pm$0.04) & -1.27 ($\pm$0.33) & -0.10 ($\pm$0.03) & -1.08 ($\pm$0.25) & -0.50 ($\pm$0.28) \\ \midrule
% \multirow{5}{*}{Robust-feature}& $\rho_n = 0.001$ & -1.34 ($\pm$0.15) & -0.12 ($\pm$0.03) & -1.26 ($\pm$0.33) & -0.09 ($\pm$0.03) & -1.07 ($\pm$0.25) & -0.49 ($\pm$0.28) \\
% & $\rho_n = 0.01$& -1.26 ($\pm$0.14) & -0.06 ($\pm$0.02) & -1.19 ($\pm$0.30) & -0.04 ($\pm$0.01) & -1.00 ($\pm$0.23) & -0.42 ($\pm$0.25) \\
% & $\rho_n = 0.1$& -0.69 ($\pm$0.08) & 0.04 ($\pm$0.01)  & -0.44 ($\pm$0.12) & 0.01 ($\pm$0.01)  & -0.33 ($\pm$0.09) & -0.09 ($\pm$0.07) \\
% & $\rho_n = 0.15$& -0.45 ($\pm$0.06) & 0.06 ($\pm$0.01)  & -0.16 ($\pm$0.06) & 0.02 ($\pm$0.00)  & -0.11 ($\pm$0.04) & -0.01 ($\pm$0.02) \\
% & $\rho_n = 0.2$& -0.29 ($\pm$0.04) & 0.07 ($\pm$0.01)  & -0.03 ($\pm$0.02) & 0.02 ($\pm$0.00)  & -0.02 ($\pm$0.01) & 0.00 ($\pm$0.00)  
% \\  \bottomrule
% \multicolumn{7}{l}{\begin{tabular}[c]{l} 
% - All results are the average values of 30 replications. Values in the parentheses are standard deviations
% \end{tabular}}
% \end{tabular}
% }
% \end{table}

\subsection{Implementation of robust DCMs in pricing problem}
DCMs are used to for many transportation decision-making problems. Pricing or fare policy design is one of the most important decision-making problems. In this section, we will test how the robust-feature MNL model performs in a typical pricing problem. Consider the Swissmetro dataset, assuming the government wishes to use the survey results to decide the best fare policy for the Swissmetro so as to maximize the revenue (i.e., price $\times$ demand). Given the same definition as Section \ref{sec_binary_exp_setup}, assuming we wish to maximize the revenue of the 1000 samples in $\mathcal{D}^{\text{Test}}$. Denote the corresponding sample index set as $\mathcal{N}_{\mathcal{D}^{\text{Test}}}$. But we only observe their perturbed attributes $\boldsymbol{\tilde{x}_n}(\alpha) = (\boldsymbol{\tilde{x}_n^{\text{NoSMCost}}}, \alpha\cdot x_n^{\text{SMCost}})$, where $\boldsymbol{\tilde{x}_n^{\text{NoSMCost}}}\in\mathbb{R}^{|\mathcal{K}|-1}$ is the perturbed attributes vector excluding the Swissmetro cost. $x_n^{\text{SMCost}}\in\mathbb{R}$ is the original cost of Swissmetro (not perturbed).
$\alpha$ is the scaler for the price of the ticket, which is our decision variable. Given an estimated parameter $\hat{\bb}$, we wish to solve the following problem to maximize the revenue:
\begin{align}\label{eq_fare_opt}
    \max_{\alpha \geq 0} \sum_{n\in\mathcal{N}_{\mathcal{D}^{\text{Test}}}}  \alpha\cdot x_n^{\text{SMCost}} \cdot\frac{\exp((\hat{\bb}_{\text{SM}})^T\boldsymbol{\tilde{x}_n}(\alpha))}{\sum_{j\in\mathcal{C}_n} \exp \left(\hat{\bb}_{j}^T\boldsymbol{\tilde{x}_n}(\alpha)\right)}
\end{align}
where $\hat{\bb}_{\text{SM}}$ is the estimated parameters for Swissmetro. The problem is a non-convex optimization. However, since there is one single decision variable $\alpha$, we can solve it through brute-force searching within $0\leq\alpha\leq10$. Assuming $\alpha^*$ is the optimal solution of Eq. \ref{eq_fare_opt}. The final actual revenue will be evaluated using the ``true'' behavior mechanism ${\bb}^{\text{Test}}$ (see details in Section \ref{sec_binary_exp_setup}) and the non-perpetuated attributes $\bx_n$.  
\begin{align}
    \text{Revenue} = \sum_{n\in\mathcal{N}_{\mathcal{D}^{\text{Test}}}} \alpha \cdot x_n^{\text{SMCost}} \cdot \frac{\exp(({\bb}^{\text{Test}}_{\text{SM}})^T\boldsymbol{{x}_n}(\alpha))}{\sum_{j\in\mathcal{C}_n} \exp \left(({\bb}^{\text{Test}}_{j})^T\boldsymbol{{x}_n}(\alpha)\right)}
\end{align}
The final results are summarized in Table \ref{tab_res_max_revenue}. ``Oracle'' scenario is obtained by feeding $\bb^{\text{Test}}$ and $\boldsymbol{{x}_n}(\alpha)$ to Eq. \ref{eq_fare_opt}, which represents the upper bound of the revenue. The results show that given better predictability, the robust models also outperform the MNL model in decision-making problems.

\begin{table}[H]
\centering
\caption{Impact of different norms on model performance}\label{tab_res_max_revenue}
\resizebox{\linewidth}{!}
{
\begin{tabular}{@{}llccc@{}}
\toprule
\textbf{Model}  & \textbf{Parameters}   & \textbf{Objective} & \textbf{Optimal solution} ($\alpha^*$) & \textbf{Revenue}  \\ \midrule
MNL & -  & 2135.70 ($\pm$247.54) & 2.51 ($\pm$3.90)  & 1931.37 ($\pm$273.12)   \\ 
{Robust-feature} & $\rho_n = 0.01$& 2105.58 ($\pm$230.14)&4.35 ($\pm$4.78)&	\cellcolor{gray!25}1985.99 ($\pm$276.94)
   \\
{Robust-label} & $\Gamma = 3$&  2076.27 ($\pm$166.01)& 3.44 ($\pm$4.46)& 1978.73 ($\pm$275.83)
  \\\midrule
{Oracle} & - &  2067.53 ($\pm$208.25)&3.40 ($\pm$4.49)&2067.53 ($\pm$208.25)
 \\
\bottomrule
\multicolumn{5}{l}{\begin{tabular}[c]{l} 
- All results are the average values of 30 replications. Values in the parentheses are standard deviations \\
\end{tabular}}
\end{tabular}
}
\end{table}

\section{Conclusion}\label{sec_conclusion}
In this paper, we propose a robust BNL and MNL model framework that accommodates testing data uncertainties. The goal is to enhance the model's prediction accuracy in new samples as a classfication problem. Our model is rooted in the theory of robust optimization. Specifically, we address feature uncertainties by assuming the $\ell_p$-norm of measurement errors are below a predetermined threshold. For label uncertainties, we assume there are at most $\Gamma$ mislabeled choices as the uncertainty set. Under these assumptions, we derive tractable robust counterparts for both robust-feature and robust-label DCMs. The proposed models are validated in both binary and multinomial choice data sets. The results demonstrate that the robust models outperform the conventional BNL and MNL models in terms of prediction accuracy and log-likelihood when there are measurement errors in testing data. Our findings suggest that the robustness component functions as ``regularization'', leading to improved generalizability of the models.

There are several future extensions for the current model. First, it is possible to combine the robust-label and the robust-feature models by assuming an uncertainty set with both feature and label measurement errors. Future research may work on deriving the closed-form formulations for the integrated model. Second, The performance of robust models may depend on the hyper-parameters. Future studies may develop methods to automatically tune the hyper-parameters. Third, the uncertainty set defined in this study is decomposable for each individual $n$ (i.e., $\mathcal{Z}(\boldsymbol{\rho}) = \prod_{n\in\mathcal{N}} \mathcal{Z}_n(\rho_n)$). One may also add constraints for the total errors $\norm{(\boldsymbol{\Delta x}_{n})_{n\in\mathcal{N}}}_p$. However, this will break the tractability for deriving the closed-form formulation (but still solvable). Future studies may explore how different definitions of the uncertainty set impact the prediction performance of robust DCMs.

\section*{Authors Contributions}
\textbf{Baichuan Mo}: Conceptualization, Methodology, Software, Formal analysis, Data Curation, Writing - Original Draft, Writing - Review \& Editing, Visualization. \textbf{Yunhan Zheng}: Conceptualization, Methodology, Software, Data Curation, Writing - Original Draft.  \textbf{Xiaotong Guo}: Conceptualization, Methodology, Writing - Review \& Editing. \textbf{Ruoyun Ma}: Software, Data Curation. \textbf{Jinhua Zhao}: Conceptualization, Supervision, Project administration, Funding acquisition.

\appendix
\appendixpage
% \section{Derivation of MNL model}\label{append_mnl}
% More details can be referred to in the DCM textbooks \citep{ben1985discrete,train2009discrete}. Recall that:
% \begin{align}
%     P_{n,i} = \mathbb{P}(U_{n,i} \geq U_{n,j}, \; \forall j\in \mathcal{C}_{n}) = \mathbb{P}(\epsilon_{n,j} \leq \epsilon_{n,i}+\boldsymbol{\beta}_j^T  \boldsymbol{x}_{n,j} - \boldsymbol{\beta}_i^T  \boldsymbol{x}_{n,i}, \; \forall j\neq i) 
% \end{align}

% where $\epsilon_{n,i}$ is assumed to be i.i.d. Gumbel distribution, with $F_{\epsilon_{n,i}}(z) = e^{e^{-z}}$, where $F_{\epsilon}(\cdot)$ is the cumulative density function (CDF). Then $\epsilon_{n,i}-\epsilon_{n,j}$ follows the logistic distribution with CDF $F_{\Delta \epsilon_{i,j}}(z) = \frac{1}{1+e^{-z}}$. Therefore, given 

\section{Inequality gap for the robust-feature MNL approximation}\label{append_ineq_gap}
Define $\Delta \boldsymbol{x}_n ^* := \argmax_{\substack{\boldsymbol{\Delta x}_n \in \mathcal{Z}_n }}\log \sum_{j\in\mathcal{C}_n}\exp \big((\boldsymbol{\beta}_j-\boldsymbol{\beta}_{I_n})^T (\boldsymbol{x}_{n}+\boldsymbol{\Delta x}_{n})\big)$. Then the gap of the Jensen's inequality in Eq. \ref{eq_jensen_ineq} is :
\begin{align}\label{eq_gap_jensen}
    \text{Gap}_n = \text{(RHS - LFS)} = \log \frac{\sum_{j\in\mathcal{C}_n}\exp  \big((\boldsymbol{\beta}_j-\boldsymbol{\beta}_{I_n})^T \boldsymbol{x}_{n} + \rho_n \cdot \norm{\bb_j - \bb_{I_n}}_q\big)}{\sum_{j\in\mathcal{C}_n}\exp \big((\boldsymbol{\beta}_j-\boldsymbol{\beta}_{I_n})^T (\boldsymbol{x}_{n}+\boldsymbol{\Delta x}^*_{n})\big)}
\end{align}
So the gap will be zero if $\rho_n \norm{\bb_j - \bb_{I_n}}_q = (\boldsymbol{\beta}_j-\boldsymbol{\beta}_{I_n})^T\boldsymbol{\Delta x}^*_{n} \quad \forall j\in\mathcal{C}_n$, which implies $\boldsymbol{\beta}_j = \boldsymbol{\beta}_i$ for all $i,j \in\mathcal{C}_n$. In the specification of DCM, one of the alternatives will have a fixed coefficient of feature $k\in\mathcal{K}$ to be zero (to avoid perfect co-linearity). Therefore, $\exists i \in \mathcal{C}$ such that ${\bb}_{i},(k) = 0, \; \forall k \in\mathcal{K}$. Then $\boldsymbol{\beta}_j = \boldsymbol{\beta}_i$ is equivalent to $\boldsymbol{\beta} = 0$. Since larger $\rho_n$ will shrink $\bb$ toward zero, we know that the gap tends to be small when $\rho_n$ is large.
Another condition where the gap is close to zero is that one $\bb_i - \bb_{I_n}$ is significantly larger than others. (i.e., $\exists i^*$ such that $(\bb_{i^*} - \bb_{I_n})^T(\boldsymbol{x}_{n}+\boldsymbol{\Delta x}_{n}) \gg (\bb_{j} - \bb_{I_n})^T(\boldsymbol{x}_{n}+\boldsymbol{\Delta x}_{n}), j\neq i^*$. In this case, we can ignore other terms in the summation and the RHS and LHS will be the same. 

Now let us find an upper bound for the gap. Notice that:
\begin{align}
    \rho_n \cdot \norm{\bb_j - \bb_{I_n}}_q \leq \rho_n\cdot \max_{i\in\mathcal{C}_n}  \norm{\bb_i - \bb_{I_n}}_q,\quad \forall i \in \mathcal{C}_n \label{eq_up_bd_1}\\
    \sum_{j\in\mathcal{C}_n}\exp \big((\boldsymbol{\beta}_j-\boldsymbol{\beta}_{I_n})^T (\boldsymbol{x}_{n}+\boldsymbol{\Delta x}^*_{n})\big) \geq \sum_{j\in\mathcal{C}_n}\exp \big((\boldsymbol{\beta}_j-\boldsymbol{\beta}_{I_n})^T \boldsymbol{x}_{n}\big)\label{eq_up_bd_2}
\end{align}
where the second inequality is obtained by setting $\boldsymbol{\Delta x}_{n} = 0$. Substituting Eqs. \ref{eq_up_bd_1} and \ref{eq_up_bd_2} to Eq. \ref{eq_gap_jensen}, we have
\begin{align}\label{eq_gap_up}
\text{Gap}_n \leq |\mathcal{C}_n| \cdot \rho_n \cdot \max_{i\in\mathcal{C}_n}  \norm{\bb_i - \bb_{I_n}}_q
\end{align}
Therefore, we have the property that when $\rho_n$ is small, the gap is also small. Combining the analysis above, the gap for the approximation will not become extreme no matter what the values of $\rho_n$ are. When $\rho_n$ is large, $\bb$ will be close to 0, which helps to lower the gap. When $\rho_n$ is small, Eq. \ref{eq_gap_up} implies the gap is small as well.

\section{Upper bound for the robust-feature MNL problem}\label{append_upper}
The upper bound of the robust MNL problem can be obtained by getting the lower bound of the LHS in Eq. \ref{eq_jensen_ineq}.
Notice that 
\begin{align}
\sum_{j\in\mathcal{C}_n}\exp \big((\boldsymbol{\beta}_j-\boldsymbol{\beta}_{I_n})^T (\boldsymbol{x}_{n}+\boldsymbol{\Delta x}_{n})\big) \geq \max_{i\in\mathcal{C}_n}\exp \big((\boldsymbol{\beta}_i-\boldsymbol{\beta}_{I_n})^T (\boldsymbol{x}_{n}+\boldsymbol{\Delta x}_{n})\big) 
\end{align}
The inequality is because the sum of all positive terms is at least as large as any single term (including the maximum one). Take the logarithm of both sides:  
\begin{align}
\log \sum_{j\in\mathcal{C}_n}\exp \big((\boldsymbol{\beta}_j-\boldsymbol{\beta}_{I_n})^T (\boldsymbol{x}_{n}+\boldsymbol{\Delta x}_{n})\big) \geq \max_{i\in\mathcal{C}_n} \big((\boldsymbol{\beta}_i-\boldsymbol{\beta}_{I_n})^T (\boldsymbol{x}_{n}+\boldsymbol{\Delta x}_{n})\big) 
\end{align}
Then we maximize over $\boldsymbol{\Delta x}_{n}$ on both sides:
\begin{align}
\max_{\substack{\boldsymbol{\Delta x}_n \in \mathcal{Z}_n }}\log \sum_{j\in\mathcal{C}_n}\exp \big((\boldsymbol{\beta}_j-\boldsymbol{\beta}_{I_n})^T (\boldsymbol{x}_{n}+\boldsymbol{\Delta x}_{n})\big) &\geq \max_{i\in\mathcal{C}_n} \big(\max_{\substack{\boldsymbol{\Delta x}_n \in \mathcal{Z}_n }}(\boldsymbol{\beta}_i-\boldsymbol{\beta}_{I_n})^T (\boldsymbol{x}_{n}+\boldsymbol{\Delta x}_{n})\big) \notag \\
& =\max_{i\in\mathcal{C}_n} \left((\boldsymbol{\beta}_{i} -\boldsymbol{\beta}_{I_n})^T \boldsymbol{x}_{n} + \rho_n \norm{\boldsymbol{\beta}_{i} -\boldsymbol{\beta}_{I_n}}_q\right)
\end{align}
Therefore, Eq. \ref{eq_inner_min_ind} can be approximated as:
\begin{align}
    \sum_{n\in\mathcal{N}} \left((\boldsymbol{\beta}_{i} -\boldsymbol{\beta}_{I_n})^T \boldsymbol{x}_{n} + \rho_n \norm{\boldsymbol{\beta}_{i} -\boldsymbol{\beta}_{I_n}}_q\right) \leq -t \quad \forall i\in\mathcal{C} 
\end{align}
Therefore, an upper bound of the robust-feature MNL problem is:
\begin{subequations}\label{eq_robust_mnl_upper}
 \begin{align}
 \max_{\boldsymbol{\beta} \in \mathcal{B},\; t} & \quad t \\
     \text{s.t.} &  \sum_{n\in\mathcal{N}} \left((\boldsymbol{\beta}_{i} -\boldsymbol{\beta}_{I_n})^T \boldsymbol{x}_{n} + \rho_n \norm{\boldsymbol{\beta}_{i} -\boldsymbol{\beta}_{I_n}}_q\right) \leq -t \quad \forall i\in\mathcal{C} 
\end{align}   
\end{subequations}

\section{Bias-variance trade off for errors without perturbations}\label{append_bia_var}

The expected errors without data perturbation can be decomposed as bias and variances:
\begin{align}
    \mathbb{E}\left[\norm{\boldsymbol{P}(\boldsymbol{x}) -\hat{\boldsymbol{P}}({\boldsymbol{x}}\mid\hat{\bb})}_1 \right] &= \mathbb{E}\left[\norm{\boldsymbol{P}(\boldsymbol{x}) - \mathbb{E}[\hat{\boldsymbol{P}}({\boldsymbol{x}}\mid\hat{\bb})] + \mathbb{E}[\hat{\boldsymbol{P}}({\boldsymbol{x}}\mid\hat{\bb})]-\hat{\boldsymbol{P}}({\boldsymbol{x}}\mid\hat{\bb})}_1 \right] 
    \\&\leq \underbrace{\norm{\boldsymbol{P}(\boldsymbol{x}) -\mathbb{E}[\hat{\boldsymbol{P}}({\boldsymbol{x}}\mid\hat{\bb})]}_1}_{\text{Bias term}}   + \underbrace{\mathbb{E}\left[\norm{\hat{\boldsymbol{P}}({\boldsymbol{x}}\mid\hat{\bb})-\mathbb{E}[\hat{\boldsymbol{P}}({\boldsymbol{x}}\mid\hat{\bb})]}_1 \right]}_{\text{Variance term}},
\end{align}
where the second inequality is derived from the triangle inequality. 
Note that typical bias-variance trade off are derived from the squared errors. For the $\ell_1$ norm, the ``variance'' term here is essentially a measure of dispersion or variability, often referred to as expected deviation or mean absolute deviation from the predictor. We keep the name of ``variance'' here to be consistent with the literature. 

For the proposed robust MNL estimator, as it is biased (Propositions \ref{prop_inconsist} and \ref{prop_inconsist_RL}). The bias term will be larger than the nominal MLE estimator. However, as we show it has higher trace of the Fisher information matrix (Proportions \ref{prop_info_matrix} and \ref{prop_info_matrix_rl}), the variance term will be smaller than the MLE estimator. Therefore, the final expected errors without data perturbation will depend on the selection of $\rho_n$. A larger value of $\rho_n$ will increase the bias and decrease the variances.

\bibliography{mybibfile}

\end{document}